\documentclass[11pt,a4paper]{article}
\usepackage[utf8]{inputenc}
\usepackage[english]{babel}
\usepackage[T1]{fontenc}
\usepackage{lmodern}
\usepackage{amsmath}
\usepackage{accents}
\usepackage{amsfonts}
\usepackage{amssymb}
\usepackage{amsthm}
\usepackage{amscd}
\usepackage{url}
\usepackage[shortlabels]{enumitem}
\usepackage{hhline}
\usepackage{dsfont}
\usepackage{array}
\usepackage{geometry}
\usepackage{tikz}
\usetikzlibrary{decorations.pathreplacing}
\usetikzlibrary{patterns}
\usepackage{xcolor}
\usetikzlibrary{arrows}
\usepackage{mdframed}
\usepackage{caption}
\usepackage[titletoc]{appendix}
\usepackage{csquotes}
\usepackage{hyperref}
\usepackage{float}
\usepackage{bm}
\usepackage{microtype}
\usepackage{graphicx}

\mdfsetup{nobreak=true}

\AtBeginDocument{
  \label{CorterrectFirstPageLabel}
  
}

\usepackage[
style=alphabetic,
maxalphanames=10,
minalphanames=5,
maxbibnames=10,
firstinits=true,
backend=biber,
doi=false,isbn=false,url=false
]{biblatex}
\newcolumntype{L}{>{$}l<{$}}

\newcommand\f{f} 
\newcommand\F{F} 
\newcommand{\U}{\mathcal{U}}
\newcommand{\BB}{\mathcal{B}}

\newcommand\N{\mathbb{N}}

\newcommand\HH{\mathcal{H}}

\newcommand\R{\mathbb{R}}

\newcommand\C{\mathbb{C}}
\newcommand\x{\underline{x}}

\newcommand\y{\underline{y}}

\newcommand\Ss{\mathcal{S}}
\newcommand\E{\mathbb{E}}

\DeclareMathOperator\Var{Var}

\newcommand\dd{\mathrm{d}}
\newcommand\CC{\mathcal{C}}

\newcommand\K{\mathcal{K}}

\newcommand\Pt{\mathcal{P}}

\DeclareMathOperator\Cov{Cov}
\DeclareMathOperator\Ker{Ker}

\newcommand\Card{\mathrm{Card}}

\DeclareMathOperator\Span{Span}
\DeclareMathOperator\Proj{Proj}

\newcommand{\enstq}[2]{\left\{#1~\middle|~#2\right\}}
\newcommand{\cdbr}[2]{\left[#1~\middle|~#2\right]}

\DeclareMathOperator\Vol{Vol}

\newcommand\one{\mathds{1}}

\newcommand\numberthis{\addtocounter{equation}{1}\tag{\theequation}}
\newcommand\jump{\par\medskip}
\newcommand{\nnjump}{}
\newcommand\quand{\quad\text{and}\quad}
\newcommand{\Dfrac}[2]{%
  \dfrac{\displaystyle #1}{\displaystyle #2}%
}
\usepackage{bbm}
\usepackage{xcolor}
\usepackage{xcolor}
\newcommand{\miknew}[1]{#1}
\newcommand{\mikcancel}[1]{}

\newcommand{\sexy}{$p$-interpolating space}
\newcommand{\horny}{$p$-interpolating space}
\newcommand{\de}{\partial}

\usepackage{mathtools} \mathtoolsset{showonlyrefs,showmanualtags}
\numberwithin{equation}{section}

\newcommand{\be}{\begin{equation}}
\newcommand{\ee}{\end{equation}}
\newcommand{\bega}{\begin{equation}\begin{aligned}}
\newcommand{\eega}{\end{aligned}\end{equation}}
\newcommand{\kop}{\left\{}
\newcommand{\pok}{\right\}}
\newcommand{\tyu}{\left(}
\newcommand{\uyt}{\right)}
\newcommand{\qwe}{\left[}
\newcommand{\ewq}{\right]}

\makeatletter
\newcommand{\tpitchfork}{%
  \raise-0.1ex\vbox{
    \baselineskip\z@skip
    \lineskip-.52ex
    \lineskiplimit\maxdimen
    \m@th
    \ialign{##\crcr\hidewidth\smash{$-$}\hidewidth\crcr$\pitchfork$\crcr}
  }%
}

\theoremstyle{plain}
\newtheorem{theo}{Theorem}[section]
\newenvironment{theorem}%
  {\begin{mdframed}[backgroundcolor=white]\begin{theo}}%
  {\end{theo}\par\vspace{0.1cm}\end{mdframed}}
  
\theoremstyle{plain}
\newtheorem{coro}[theo]{Corollary}
\newenvironment{corollary}%
  {\begin{mdframed}[backgroundcolor=white]\begin{coro}}%
  {\end{coro}\par\smallskip\end{mdframed}}
  
\theoremstyle{plain}
\newtheorem{lemm}[theo]{Lemma}
\newenvironment{lemma}%
  {\begin{mdframed}[backgroundcolor=white]\begin{lemm}}%
  {\end{lemm}\par\vspace{0.cm}\end{mdframed}}

\theoremstyle{plain}
\newtheorem{prop}[theo]{Proposition}
\newenvironment{proposition}%
  {\begin{mdframed}[backgroundcolor=white]\begin{prop}}%
  {\end{prop}\par\vspace{0.cm}\end{mdframed}}
  
\theoremstyle{definition}
\newtheorem{defn}[theo]{Definition}
\newenvironment{definition}%
  {\begin{mdframed}[backgroundcolor=white]\begin{defn}}%
  {\end{defn}\par\vspace{0.cm}\end{mdframed}}

\newenvironment{acknowledgements}{%
  \begin{abstract}
}{%
  \end{abstract}
}

\theoremstyle{definition}
\newtheorem{remark}[theo]{Remark}

\makeatletter
\def\blfootnote{\gdef\@thefnmark{}\@footnotetext}
\makeatother
  
\makeatletter
\renewcommand*\env@matrix[1][*\c@MaxMatrixCols c]{%
  \hskip -\arraycolsep
  \let\@ifnextchar\new@ifnextchar
  \array{#1}}
\makeatother

\begin{document}

\title{\Huge{The number of critical points of a Gaussian field: finiteness of moments}}
\author{Louis Gass, Michele Stecconi}
\maketitle
\blfootnote{University of Luxembourg}
\blfootnote{This work was supported by the Luxembourg National Research Fund (Grant: 021/16236290/HDSA)}
\blfootnote{Email: louis.gass(at)uni.lu,\quad michele.stecconi(at)uni.lu}

\begin{abstract}
Let $f$ be a Gaussian random field on $\R^d$ and let $X$ be the number of critical points of $f$ contained in a compact subset. 
A long-standing conjecture is that, under mild regularity and non-degeneracy conditions on $f$, the random variable $X$ has finite moments. So far, this has been established only for moments of order lower than three. In this paper, we prove the conjecture. Precisely, we show that $X$ has finite moment of order $p$, as soon as, at any given point, the Taylor polynomial of order \miknew{$p$} of $f$ is non-degenerate. We present a simple and general approach that is not specific to critical points and we provide various applications. In particular, we show the finiteness of moments of the nodal volumes and the number of critical points of a large class of smooth, or holomorphic, Gaussian fields, including the Bargmann-Fock ensemble.
\end{abstract}

\renewcommand\contentsname{} 

\begingroup
\let\clearpage\relax
\vspace{-1cm} 
\setcounter{tocdepth}{2}
\tableofcontents
\endgroup
\section{Introduction}
The study of the nodal set associated with a random field has a long history, and is in particular motivated by the pioneering works of Kac and Rice, see e.g. \cite{Aza09} for a general introduction to this topic. Classical nodal observables include for instance the nodal volume of random fields, the number of nodal components and other related topological quantities.\jump

The exact distribution of a random nodal set is out of reach, and a great amount of research focuses on the asymptotic behavior of the expectation and the variance of nodal quantities, together with associated central and non-central limit theorems, either on a growing domain or as some parameter goes to infinity (as in the random wave model).  We refer for instance to the papers \cite{Zel09,Kri13,Can16,Mar16,Nou19,Die20,Gas21, LETEuler} for random waves asymptotics, \cite{Naz09, Bel18, Bel22} for asymptotics of excursion sets, and \cite{Gay17,GaWe2,stec2019MaxTyp} for higher Betti numbers.\jump

In this context, the study of the moments associated with a random nodal quantity appears as a convenient tool in order to gather information on its distribution. The celebrated Kac--Rice formula (see \cite{Kac43,Ric45}) gives an integral expression for the expectation, the variance and higher moments of the nodal volume of a random field satisfying some mild regularity and non-degeneracy conditions. Moreover, if one is able to compute the asymptotic of all moments, then one can recover limit theorems for the nodal volume by the method of moments, as in \cite{Naz12,Bla19,Anc21c,Gas21t, AL_rootsKost}.\jump

Unfortunately, establishing general conditions on a random field ensuring the finiteness of the $p$-th moment of its nodal volume turns out to be a delicate problem. A mild criterion has been obtained in \cite{Arm23} when the random field takes values in $\R$. The proof is an application of Hermite-Lagrange interpolation, which is a tool only available in dimension $1$, and it does not seem to be adaptable to higher dimension. For moments of order $2$ (see \cite{Bel19,Aza22,Lad22}) one can prove finiteness of the second moment by means of the Kac--Rice formula, complemented by an in-depth analysis of the attraction/repulsion between maxima, minima and saddle points. Recently, the article \cite{Bel22} proved the finiteness of the third moment for the number of critical points of a regular non-degenerate Gaussian field taking values in $\R$, by proving the integrability of the Kac density. The proof relies on a technical Taylor expansion near the singularity of the Kac density and does not seem to be easily adapted to the analysis of higher moments. At last, let us mention the article \cite{Mal94}, that proves the finiteness of every moments for the number of intersection points of two independent Gaussian fields from $\R^2$ to $\R$, again by a technical analysis of the Kac density.\jump

In this paper, we prove the finiteness of the $p$-th moment of the number of critical points of a Gaussian random field on a compact subset, only assuming regularity and the non-degeneracy of the Taylor polynomial expansion up to order \miknew{$p$}. In particular, we recover some of the aforementioned results for the finiteness of low moments. The proof relies on Kergin interpolation \cite{KERGIN}, which is a form of multivariate polynomial interpolation that is well-suited for the analysis of the Kac density. Heuristically, it allows us to pass from a general Gaussian field to a random polynomial field, for which the finiteness of the number of critical points is a direct consequence of Bezout's theorem.\jump

The proof is in fact quite robust and we were able to prove the finiteness of the $p$-th moment of the number of zeros of a large class of Gaussian random field satisfying some mild non-degeneracy assumption, including gradient random fields, random fields with independent coordinates, analytic random fields, etc.
\subsubsection*{Organization of the paper}
The paper is organized as follows. In the introduction we state our main theorems about the finiteness of moments of the number of zeros of a Gaussian random field and provide a heuristic proof. We also develop an application to the Bargmann-Fock Gaussian field and some extensions to a manifold setting, as well as to nodal volume in higher dimensions. In a second part, we introduce concepts related to multivariate interpolation, and in particular Kergin interpolation, which plays a central role in our proof. The last section is devoted to proofs of the main theorems and related auxiliary results.

\subsection{Statement of the mains results}
Let $d$ be a positive integer and $\U$ be an open subset of $\R^d$. For a function $\F:\U\mapsto\R^d$ and a compact subset $K$ of $\U$ we define
\[Z(\F,K) := \enstq{x\in K}{F(x) = 0}.\]
Using the standard multi-index notations, we define for a multi-index $\alpha\in \N^d$ the operator
\[\partial^\alpha = \partial_1^{\alpha_1}\ldots\partial_d^{\alpha_d},\]
that act on functions of class $\CC^{|\alpha|}$ on $\U$. Given a set $Z$, we denote by $\#Z$ its cardinality. In the following, we always identify $\C$ with $\R^2$ and, for a multi-index $\alpha\in\N^d,$ we still denote $\partial^\alpha$ the standard complex differentiation that acts on holomorphic functions on an open subset of $\C^d$. 
\subsubsection{The two main theorems}
We state the two main theorems of this paper about the finiteness of moments of the number of zeros of a random Gaussian field.
\nnjump
\begin{theorem}
\label{thm1}
Let $p,d$ be positive integers, $\U$ be an open subset of $\R^d$ (resp. $\C^d$) and $F$ be a Gaussian random field from $\U$ to $\R^d$ (resp. $\C^d$) with a.s. $\CC^p$ (resp. holomorphic) sample paths. Assume that, for every $x\in \U$, the random Gaussian vector \miknew{$(\partial^\alpha F(x))_{|\alpha|\leq p-1}$} is non-degenerate. Then for every compact subset $K$ of $\U$,
\[\E\qwe\#Z(F,K)^p\ewq<+\infty.\]
\end{theorem}

\nnjump
\begin{theorem}
\label{thm2}
Let $p,d$ be positive integers, $\U$ be an open subset of $\R^d$ (resp. $\C^d$) and $f$ be a Gaussian random field from $\U$ to $\R$ with a.s. $\CC^{p+1}$ (resp. holomorphic) sample paths. Assume that, for every $x\in \U$, the random Gaussian vector \miknew{$(\partial^\alpha f(x))_{|\alpha|\leq p}$} is non-degenerate. Then for every compact subset $K$ of $\U$,
\[\E\qwe\#Z(\nabla f,K)^p\ewq<+\infty.\]
\end{theorem}
This last Theorem \ref{thm2} proves the finiteness of the $p$-th moment of the number of critical points of a regular non-degenerate real or complex Gaussian field. Note that Theorem \ref{thm2} is not a consequence of Theorem \ref{thm1} because the symmetry of the second derivatives prevent the Gaussian vector \miknew{$(\partial^\alpha \nabla f)_{|\alpha|\leq p-1}$} from being non-degenerate.
However, we will prove both theorems with the same method, as a consequence of a general and more abstract result, Theorem \ref{thm:general}, which we discuss below.  \jump

\begin{remark}\label{rem:uniform}
The constant that bounds the $p$-th moment for the number of random zero in Theorem \ref{thm1} is a bounded functional of the associated Gaussian field as a $\CC^p$ (resp. holomorphic) random function, whose expression is directly related to Kac--Rice formula. In particular, under the setting of Theorem \ref{thm1}, if $(F_n)_{n\geq 0}$ is a sequence of Gaussian fields that converges to $F$ for the $\CC^p$ topology (resp. topology of uniform convergence) on $\U$, then 
\[\lim_{n\rightarrow+\infty} \E\qwe\#Z(F_n,K)^p\ewq = \E\qwe\#Z(F,K)^p\ewq,\]
and a similar statement holds for Theorem \ref{thm2}.
\end{remark}
\jump

\subsubsection{Improvement of the central limit theorem in \cite{Bel22}}
The real version of Theorem \ref{thm2} generalizes to any integer $p$ the content of \cite[Thm. 1.6]{Bel22}, yielding a third moment bound on the number of critical points of a smooth and non-degenerate random field. This is a key point in their proof of the central limit theorem for the number of excursion sets, since the number of critical points provides an easy upper bound for the number of excursion sets. \jump

More precisely, for $l\in\R$, $R>0$ and a function $f:\R^d\rightarrow\R$, we denote by $N(R,l)$ the number of connected components of either the set $\{f=l\}$ or the set $\{f\geq l\}$, that are totally contained in the box $\Lambda_R = ]-R,R[^d$. Now let $f$ be a Gaussian random field from $\R^d$ to $\R$ with spatial moving representation
\[f = q\star W,\]
where $q$ is a symmetric function and $W$ is a Gaussian white noise on $\R^d$. Note that the covariance function of the process $f$ has then the expression $q\star q$. According to \cite[Rem. 3.9]{Bel22}, Theorem \ref{thm2} of this present paper leads to a weaker decay assumption for the covariance function of a Gaussian field in \cite[Thm. 1.2]{Bel22} (under a stronger regularity hypothesis), in order to get a central limit theorem for the number of excursions sets.\jump
\begin{corollary}
Assume that
\begin{itemize}
\item $\forall \alpha\in\N^d,\quad\partial^\alpha q \in L^2(\R^d)$
\item There exists $\beta>3d$ and $c\geq 1$ such that for all $|x|\geq 1$,
\[\max_{|\alpha|\leq 2} |\partial^\alpha q(x)|\leq \frac{1}{|x|^\beta}.\]
\end{itemize}
Then there exists a constant $\sigma(l)$ such that
\[\lim_{R\rightarrow +\infty} \frac{\Var(N(R,l))}{\Vol(\Lambda_R)}=\sigma(l),\]
and
\[\frac{N(R,l)-\E[N(R,l)]}{\sqrt{\Vol(\Lambda_R)}}\underset{R\rightarrow+\infty}{\overset{d}{\longrightarrow}} \mathcal{N}(0,\sigma(l)).\]
\end{corollary}
\begin{proof}
See \cite[Thm. 1.2]{Bel22} and \cite[Rem. 3.9]{Bel22} to replace the constant $9d$ by $3d$.
\end{proof}
\subsubsection{Application to the Bargmann-Fock random field}\label{sec:BF}
The above Theorems \ref{thm1} and \ref{thm2} can be directly applied, for instance, to the real or complex Bargmann-Fock random field, that is defined as follows.  Let $(\gamma_\alpha)_{\alpha\in\N^d}$ be a family of i.i.d. normal variables and let $\psi\colon \R^d\to \R$ be the random field defined as
\be 
\psi(x):=\sum_{\alpha\in\N^d}\frac{\gamma_\alpha}{\sqrt{\alpha!}} x^\alpha,
\ee
with the standard multi-index notation.
The field $\psi$ is a Gaussian analytic function. Replacing the collection real Gaussian random variables $(\gamma_\alpha)_{\alpha\in\N^d}$ by complex Gaussian random variables, the same formula defines the Gaussian entire function $\psi_\C\colon \C^d\to \C,$ considered for instance in \cite{SodinTsirelson}.
The real and complex Bargmann-Fock fields are respectively defined for $x\in\R^d$ and $z\in\C^d$ by
\be 
\varphi(x)=\psi(x)e^{-\frac{\|x\|^2}{2}} \quad \text{and}\quad \varphi_\C(z)=\psi_\C(z)e^{-\frac{\|z\|^2}{2}}.
\ee 
It follows for instance from \cite[p. 64]{Aza09} that the fields $\varphi$ and $\varphi_\C$ satisfy the hypotheses of Theorem \ref{thm2} for all positive integers $p$. We deduce for these two fields the finiteness of all moments for the number of critical points in a compact set.\jump

Similarly, let $F = (f_1,\dots,f_d)$ be a collection of $d$ independent real or complex Bargmann-Fock random fields. Then the Gaussian field $F$ satisfies the hypotheses of Theorem \ref{thm1} for all positive integer $p$ and we deduce the finiteness of all moments for the number of intersection points in a compact set of $d$ independent copies of the real or complex Bargmann-Fock random fields.\jump
 
Note that Theorem \ref{thm1} is not applicable to the Berry random wave model on $\R^d$ (see \cite{Ber77}). Indeed this Gaussian field a.s. satisfies the partial differential equation
\[\Delta f +f=0,\]
which prevent the collection of partial derivatives at a point $x\in\R^d$ from being a non-degenerate Gaussian vector. The adaptation of the proof to this model is still a work in progress.
\subsubsection{Extension to manifolds and higher dimensions}
\newcommand{\chart}{\phi}
\newcommand{\vol}{\mathrm{vol}}
Let $p,d$ be positive integers and $0\leq n<d$ be an integer. In the following, $(M,g)$ is a smooth (resp. complex) Riemannian manifold of dimension $d$. For a function $F\colon M\to \R^{d-n}$ and a compact subset $K$ of $M$ we define 
\[Z(F,K):=\enstq{x\in K}{F(x)=0}\]
Under the hypotheses that the function $F$ is of class $\CC^1$ and that $0$ is a regular value for $F$, the set $Z(F,M)$ is a $\CC^1$ submanifold of $M$ of dimension $n$. We then denote by
\be 
\vol^{n}_g(Z(F,K))=\int_{Z(F,M)}1_K d\vol^n_g 
\ee
its Riemannian $n$-volume measure relative to the metric induced by the inclusion in $(M,g).$ A chart $\chart$ of $M$ is a smooth (resp. holomorphic) function $\chart$ from an open subset $\U$ of $M$ to $\R^d$ (resp. $\C^d$) which is a diffeomorphism onto its image. For every multi-index $\alpha\in \N^d,$ we define the differential operator
\be 
\de^\alpha_\chart f:=\de^\alpha (f\circ \chart^{-1})\circ \chart
\ee
that acts on function of class $\CC^{|\alpha|}$ (resp. holomorphic function) on $\U$. Then Theorem \ref{thm1} and Theorem \ref{thm2} generalize to this more general setting.
\nnjump
\begin{theorem}
\label{thm3}
Let $p,d$ be positive integers $0\leq n<d$ be an integer. Let $M$ be a smooth (resp. complex) Riemannian manifold of dimension $d$ and $F$ be a Gaussian random field from $M$ to $\R^{d-n}$ (resp. $\C^{d-n}$) with a.s. $\CC^p$ (resp. holomorphic) 
sample paths. Assume that for every $x\in M$, there exists a chart $\chart$ defined around $x$ such that the random Gaussian vector \miknew{$(\partial^\alpha_{\chart} F(x))_{|\alpha|\leq p-1}$} is non-degenerate. Then for every compact subset $K\subset M$,
\[\E\qwe\vol_g^{n}\tyu Z(F,K)\uyt^p\ewq<+\infty
\qquad 
\tyu\text{resp. } 
\E\qwe\vol_g^{2n}\tyu Z(F,K)\uyt^p\ewq<+\infty
\uyt
.\]
\end{theorem}
\begin{theorem}
\label{thm4}
Let $p,d$ be positive integers, $M$ be a smooth (resp. complex) 
manifold of dimension $d$ and $f$ be a Gaussian random field from $M$ to $\R$ with a.s. $\CC^{p+1}$ (resp. holomorphic) 
sample paths. Assume that for every $x\in M$, there exists a chart $\chart$ defined around $x$ such that the random Gaussian vector \miknew{$(\partial^\alpha_{\chart} f(x))_{|\alpha|\leq p}$} is non-degenerate. Then for every compact subset $K\subset M$,
\[\E\qwe\# Z(\nabla f,K)^p\ewq<+\infty.\]
\end{theorem}
We will show that Theorem \ref{thm3} and \ref{thm4} follow from Theorem \ref{thm1} and Theorem \ref{thm2}.
\nnjump
\begin{remark} 
Theorem \ref{thm3} also holds for Gaussian sections of a vector bundle and when the volume is replaced with a random variable of the form $\int_{Z(F,K)}Bd\vol_g^n,$ for some bounded measurable function $B\colon \U\to \R.$ Moreover, a similar statement to that in Remark \ref{rem:uniform} holds for Theorem \ref{thm3} and Theorem \ref{thm4}.
\end{remark}
\subsection{Heuristic of proof}

The following discussion gives a heuristic of proof for the real version of Theorem \ref{thm1}. The complex version and the case of critical points will be discussed in the next subsection, which extends the following heuristic to a more general framework.\jump

In the following, $d$ and $p$ are positive integers and $\U$ is a open subset of $\R^d$. We denote by $\CC^p(\U,\R^d)$ the class of $p$ times continuously differentiable functions from $\U$ to $\R^d$. We define the \textit{large diagonal} $\Delta$ of $\U^p$ by
\[\Delta = \enstq{(x_1,\ldots,x_p)\in \U^p}{\exists\,i\neq j\;\text{  s.t.  }x_i=x_j}.\numberthis\label{diag}\]
Given $x\in U$ and $F\in \CC^1(\U,\R^d)$, we denote by $J_xF$ the Jacobian determinant of $F$ at point $x$. We introduce the $p$-factorial power of a number $\alpha\in\R$ as
\[\alpha^{[p]} = \alpha(\alpha-1)\ldots(\alpha-p+1).\]
We first recall the celebrated Kac--Rice formula. 
\nnjump
\begin{theorem}[Kac--Rice formula]
\label{thm03}
Let $\U$ be an open subset of $\R^d$. Let $\F: \U\mapsto \R^d$ be a Gaussian random field of class $\CC^1$ such that for distinct points $y_1,\ldots,y_p$, the Gaussian vector $(F(y_1),\ldots,F(y_p))$ of size $dp$ is non-degenerate. Then for any compact subset $K$ of $U$ one has
\be\label{eq:kr}\E\qwe\#Z(F,K)^{[p]}\ewq = \int_{K^p\setminus \Delta} \rho_F(\y)\dd \y\,\ee
where for $\y\in \U^p\setminus \Delta$
\[\rho_F(\y) = \E\cdbr{\prod_{k=1}^p |J_{y_k}F|}{F(y_1)=\ldots=F(y_p) = 0}\psi_{\delta_{\y} F},\]
and $\psi_{\delta_{\y} F}$ is the density of the vector $(F(y_1),\ldots,F(y_p))$.
\end{theorem}
\begin{proof}
See \cite[Thm. 6.3]{Aza09} and Section \ref{rem:transversality} of this paper.
\end{proof}

The hypotheses on $F$ in the above Kac--Rice formula ensure that the function $\rho_F$ is well-defined on $\U^p\setminus\Delta$. The strategy adopted in this paper is to show that under suitable conditions on the Gaussian field $F$, the Kac density $\rho_F$ is integrable on $K^p\setminus \Delta$, for any compact subset $K$ of $\U$. The difficulty resides in understanding the behavior of the Kac density near the large diagonal $\Delta$.\jump

On the large diagonal, the expression of the function $\rho_F$ is ill-defined, since both the numerator and the denominator converge to zero as one approaches $\Delta$. 
The usual trick introduced in \cite{Cuz75} to lift the degeneracy of the Gaussian vector $(F(y_1),\ldots,F(y_p))$ when the points $y_1,\ldots y_p$ collapse is based on the notion of divided differences. Let us illustrate this concept the concept with two points $x$ and $y$ in $\R^p$. In that case
\[\rho_F(x,y) = \Dfrac{\E\cdbr{|J_xF||J_yF|}{F(x)=F(y) = 0}}{\sqrt{\det 2\pi\Cov((F(x),F(y))}}.\]
Let $(\vec{e},\vec{e_2},\dots \vec{e_d})$ be any orthonormal basis of $\R^d$ such that $x = y + \|x-y\|\vec{e}$. We denote by $\partial_{\vec{e}}$ the partial derivative operator in direction $\vec{e}$, and by $\nabla_{\vec{e}^\perp}$ the gradient operator on its orthogonal complement $(\vec{e_2},\ldots\vec{e_d})$. Then,
\bega 
\det \Cov(F(x),F(y)) &= \|x-y\|^{2d} \Cov\left(F(x),\frac{F(y)-F(x)}{\|x-y\|}\right)
\\
&\simeq \|x-y\|^{2d} \Cov\left(F(x),\partial_{\vec{e}}F(x)\right).
\eega
and, conditionally on $F(x) = F(y) = 0,$ one has by a Taylor expansion the approximation
\[
J_yF = \det
\begin{pmatrix} 
\partial_{\vec{e}} F(y)
 & \nabla_{{\vec{e}}^\perp} F(y)
 \end{pmatrix}
\simeq 
 \frac{\|x-y\|}{2}\det\begin{pmatrix} 
\partial^2_{\vec{e}} F(y)
 & \nabla_{{\vec{e}}^\perp} F(y)
 \end{pmatrix},
\]
Gathering both estimates,
\be
\rho_F(x,y) \simeq \frac{\|x-y\|^{2-d}}{4}\;\Dfrac{\E\cdbr{\det\begin{pmatrix} 
\partial^2_{\vec{e}} F(x)
 & \nabla_{{\vec{e}}^\perp} F(x)
 \end{pmatrix}^2}{F(x)=\partial_{\vec{e}}F(x)= 0}}{\sqrt{\det 2\pi\Cov((F(x),\partial_{\vec{e}}F(x))}}.
 \ee
Under a non-degeneracy assumption on the random field $F$ and its partial derivatives, the behavior of the function $\rho_F$ is entirely determined by the function $\|x-y\|^{2-d}$, which is locally integrable on $\R^d\times\R^d$ by a polar change of coordinates. For $3$ points, one can carry a similar -but way more technical- computation to show the local integrability of the function $\rho_F$, as done in \cite{Bel22}. For $4$ points and more, one can extends this trick only when the dimension $d=1$. In that case, a corollary of Hermite--Lagrange interpolation shows that when $y_1,\ldots y_p$ are distinct real numbers that converge to a single point $y$,
\[\det \Cov(F(y_1),\ldots,F(y_p)) = \left(\prod_{k\neq l}(y_k-y_l)\right)\Cov(F(y),F'(y),\ldots\ldots,F^{(p-1)}(y)),\]
see \cite{Gas21t,Arm23, Anc21c} for more details on the divided difference trick and Hermite--Lagrange interpolation. Unfortunately, multivariate interpolation is not as well-behaved as univariate interpolation and a direct approach to understand the behavior of the Kac density as its arguments collapse seems intractable as soon as $d>1$.\jump

The underlying heuristic is the following. In dimension $1$, there is only one direction and thus only one way for a collection of $p$ points to collapse, whereas in greater dimension there are many different possibilities for a collection of $p$ points to collapse, depending on their relative configuration. A direct approach to solve this issue is to find an appropriate compactification of $\U^p\setminus\Delta$, the configuration space of $p$ points. This is a classical topic in geometry, with a rich history and several solutions, see \cite{Sinha, Eke14} and references therein. Then one has to find a way to combine the Kac-Rice formula with a convenient compactification. Although it would give a much deeper understanding of the results of this paper, we won't pursue this route.
\jump

\miknew{\begin{remark}
Shortly after the appearence of the first version of this paper, a result in such direction was independently announced by M. Ancona and T. Letendre. 
In their work \cite{AL_cursedpaper} they define the notion of \emph{multijet}, serving as a substitute of the divided difference scheme in arbitrary dimension, and use it to prove the existence of a suitable completion $C_p[\R^d]\supset (\R^d)^p\setminus\Delta$ of the configuration space, obtained by means of a sequence of blow-ups.  Its key property is that, provided that $F\colon \R^d\to \R^{d}$ is a Gaussian field with non-degenerate Taylor polynomial of degree $p-1$ at every point, it follows that the Gaussian field $F^{\times p}:(y_1,\dots,y_p)\mapsto (F(y_1),\dots,F(y_p))$, defined on 
$(\R^d)^p\setminus\Delta$, can be extended to a non-degenerate Gaussian field $mj_p(F,\cdot)$ on $C_p[\R^d]$, with the same zeros as $F^{\times p}$ (actually, $mj_p(F,\cdot)$ is a Gaussian section of a certain vector bundle). This implies that the Kac-Rice formula for the expected number of zeros of $mj_p(F,\cdot)$ in a compact subset $\overline{K^p\setminus \Delta}$ of $C_p[\R^d]$ is finite and also that in fact it computes the factorial $p^{th}$ moment of $\#Z(F,K)$. Using this argument, the authors deduce various results, analogous to our Theorems \ref{thm1}, \ref{thm2} and \ref{thm3}
, all reported in \cite{AL_cursedpaper}. Their work and ours are independent and the proof is different.
\end{remark}
}
\jump

On the other hand, our approach has the advantage of avoiding all the geometric subtleties.
Indeed we found an elementary way to treat the singularity at the diagonal $\Delta$, which seems a bit more general since we bypass the exact description of the possible configurations of points as they collapse. The underlying idea is that for two sufficiently non-degenerate random processes $F_1$ and $F_2$, the associated Kac densities $\rho_{F_1}$ and $\rho_{F_2}$ have the same behavior near the diagonal $\Delta$. It is then sufficient to prove the integrability of the Kac density for one particular random process. The choice of a collection of random Gaussian polynomial allows us to conclude. Indeed, by Bezout's theorem, the number of common roots of $d$ generic $d$-variate polynomials is bounded by a universal constant, namely the product of the degrees of each polynomial.\jump

In order to prove this "universal" near-diagonal behavior, one has to gather exactly which quantity appears in the Kac--Rice formula. Given a point $\y\in\U^p\setminus\Delta$ we only need to know the value of $\F$ and its gradient at points $y_1,\ldots,y_p$. In other words, one needs to understand the random process $\F$ when tested against the linear forms $\delta_{y_k}$ and $\partial_i\delta_{y_k}$ for $1\leq k\leq p$ and $1\leq i\leq d$ that live in the dual space of smooth functions. The first step is to reduce this space to a finite dimensional space, namely the dual space of $d$-variate polynomials with total degree lower than $p$, denoted by $\Pt_p^d(\R)$, and exploit the compactness properties of this finite dimensional space. \jump

To this end, one can consider a form of multivariate interpolation known as \emph{Kergin interpolation} \cite{KERGIN}, defined in Section \ref{sec:kergin}. It has the advantage, in contrast to other forms of multivariate interpolation schemes, to be well-defined for any configuration of points. For a collection of points $x_1,\ldots,x_p$ with possible repetitions, Kergin's interpolation provides a projector $\Pi_{\x}^{\otimes d}$ from the space of all $\CC^{p-1}$ vector fields onto the space $(\Pt_{p-1}^d(\R))^d$ of polynomial vector fields, having the property that
\[\forall\; 1\leq k\leq p,\quad\Pi_{\x}^{\otimes d}F(x_k) = \F(x_k).\]
When some of the points collapse, the partial derivatives are also interpolated. In particular, when $x_1=x_2=\ldots=x_p$ then $\Pi_{\x}^{\otimes d}\F$ is the Taylor polynomial expansion of $F$ of degree $p-1$ at the common value. Then the non-degeneracy hypothesis on $F$ in Theorem \ref{thm1} is equivalent to the fact that locally around a point $x\in \U$, the Kergin interpolant $\Pi_{\x}^{\otimes d}\F$ is a non-degenerate collection of $d$ random $d$-multivariate polynomial.\jump

From this point, the proof is direct. For $\y\in\U^p\setminus\Delta$, we want to write
\[\rho_\F(\y) = R(\y)\sigma_\F(\y),\]
where $R$ is a universal function that captures the diagonal singularity and $\sigma_\F$ is a function that depends on the distribution of $\F$, and that is bounded above and below by positive constants on $\U\setminus\Delta$. This decomposition is achieved by applying the Gram--Schmidt orthonormalization procedure on the family of linear forms $(\delta_{y_k})_{1\leq k\leq p}$, as well as a renormalization of the determinant $J_{y_k}$ on $\Pt_p^d(\R)$. The function $R$ is a functional of the associated transformation matrices and depends only on the scalar product defined on $(\Pt_{p-1}^d(\R))^d$, whereas the residue function $\sigma_F$ is positively bounded thanks to the non-degeneracy assumption the random field.\jump

\subsection{A general statement}\label{sec:genstat}

As mentioned before, the above proof does not work if the random field $F$ is degenerate. For instance, if $F$ is the gradient field of a non-degenerate Gaussian random field $f:\U\rightarrow \R$ then the symmetry of the partial derivatives prevents the Kergin interpolant $\Pi_{\x}^{\otimes d}F$ from being a non-degenerate collection of $d$ random $d$-multivariate polynomials. Nevertheless, one can prove that the Kergin projection of a gradient field is the gradient of a polynomial (see Lemma \ref{lem:gradientKergin}). This means that the Kergin interpolant $\Pi_{\x}^{\otimes d}(\nabla f)$ is non-degenerate on the space of gradients of polynomials $\nabla\Pt_p^d$, and the proof adapts verbatim when replacing the polynomial space $(\Pt_{p-1}^d)^d$ with the polynomial space $\nabla\Pt_p(\R)^d$. This leads to the proof of Theorem \ref{thm2}.\jump

Similarly, for the complex version, one can see a holomorphic field $F:\U\rightarrow\C^d$ as a real field $F:\U\rightarrow\R^{2d}$ satisfying the Cauchy-Riemann equations. One can show in a similar fashion that the Kergin interpolating polynomial, as an element of $(\Pt_{p-1}^{2d}(\R))^{2d}$, is in fact a complex $d$-variate polynomial, that is an element of $(\Pt_{p-1}^d(\C))^d$. These two observations led us to write an abstract form of the theorem that englobes both situations. We replace the Kergin projector
\[\Pi_{\x}^{\otimes d} : \CC^p(\U,\R^d)\longrightarrow (\Pt_{p-1}^d(\R))^d\]
with a more general form
\[K_{\x} : W\longrightarrow V,\]
where for all $x\in \U^p$, the mapping $K_{\x}$ is a surjective mapping from a subspace $W\subset \CC^0(\U,\R^d)$ of functions from $\U$ to $\R^d$, to a finite dimensional subspace $V\subset \CC^0(\U,\R^d)$. The rest of this section is dedicated to a statement of sufficient conditions on $V$ and on a random field taking values in $W$ in order to recover the finiteness of the $p$-th moment for the number of random zeros.
\jump

In the following, $V$ is a vector space. We denote by $\Pt[V]$ the space of all real valued polynomial functions on $V$. If $B$ is a basis of $V$ then one can see $\Pt[V]$ as the polynomial ring with indeterminates the elements of $B$. This space is naturally endowed with a gradation corresponding to the total degree of the associated polynomial. For $p\geq 0$, we denote by $\Pt_p[V]$ the subspace of $\Pt[V]$ of polynomial with total degree lower than $p$. When $V=\R^d$ then $\Pt_p[V]$ coincide with the polynomial space $\Pt_p^d$ introduced earlier.\jump

In the following, we let $p,q$ be positive integers and $\U$ be an open subset of $\R^d$. Let $V$ be a finite dimensional subspace of the space $\CC^q(\U,\R^d)$. For a point $x\in \R^d$, we define the Dirac evaluation map at point $x$ as the mapping $\delta_x(F)= F(x).$ If $\x = (x_1,\ldots,x_p)$ are $p$ points in $\R^d$, we define
\bega 
\delta_{\x}:\CC^0(\U,\R^d)&\longrightarrow (\R^d)^p\\
F&\longrightarrow \tyu F(x_1),\dots, F(x_p)\uyt
\eega
One can see the mapping $\delta_{\x}$ as a collection of $dp$ linear forms on $\CC^0(\U,\R^d)$. We also define the Jacobian determinant
\begin{align*}
J_x:\CC^1(\U,\R^d)&\longrightarrow \R\\
F&\longrightarrow \det \nabla\! F(x)
\end{align*}
According to the previous paragraph, one can see $J_x$ as a polynomial of degree $d$ on $\CC^1(\U,\R^d)$. The following definition ensures that a finite dimensional space $V$ of $\CC^q(\U,\R^d)$ is a relevant interpolation space for the problem of interpolating a function at $p$ distinct points.

%
\nnjump
\begin{definition}[\sexy{}]
\label{def:interpolating}
Let $V$ be a finite dimensional subspace of $\CC^1(\U,\R^d)$. The space $V$ is said to be a \emph{\sexy{}} if it satisfies the following three properties, for any set of $p$ distinct points $y_1,\ldots,y_p$ in $\U$.
\begin{enumerate}[A]
\item\label{a} \miknew{There is a subspace $V_0\subset V$, such that the family $\delta_{\y}|_{V_0}$ of elements of $V_0^*$ is free. Equivalently, $\delta_{\y}(V_0)=(\R^{d})^p.$}
\item\label{b} The polynomials $(J_{y_k})_{1\leq k\leq p}$ are non-zero on the subspace $\Ker(\delta_{\y})\cap V$
\item\label{c} For every compact subset $K$ of $U$, there is a constant $C_K$ and a negligible subset $N_K$ of $V$ such that
\[\forall G\in V\setminus N_K,\quad \Card \enstq{x\in K}{G(x) = 0} \leq C_K.\]
\end{enumerate} 
\end{definition}
\nnjump
Let $G$ be a non-degenerate Gaussian vector taking value in $V$. Property $A$ implies that elements in $V$ can take arbitrary values on points $y_1,\ldots,y_p$, which is required to write the Kac--Rice formula for $G$. Property $B$ implies that the Jacobian determinant at a point $y_k$ is not determined by the evaluation at points $y_1,\ldots,y_p$ and ensures that the Kac--Rice formula for $G$ never vanishes. At last, property $C$ ensures that the random variable $\#Z(G,K)$ is bounded, hence has finite moments of every order, even though the Kac density $\rho_G$ is unbounded near the diagonal.  It mimics the property that the number of common roots of $d$ generic $d$-multivariate polynomial is bounded by a universal constant, namely the products of their degrees, by Bezout's Theorem (see Section \ref{sec.bezout}). Indeed, in the concrete cases that we consider, $V$ is always a space of polynomials. \jump

Now let $W$ be a linear subspace of $\CC^q(\U,\R^d)$. The following definition ensures that $V$ is a relevant interpolation space for functions in $W$.
\nnjump
\begin{definition}[adapted \horny{}]\label{def:interpolator}
Let $W$ be a closed subspace of $\CC^q(\U,\R^d)$.
The space $V$ is said to be a \emph{\horny{} adapted to $W$}  if the three following properties are satisfied :
\begin{enumerate}
\item\label{4} $V$ is a \sexy{}.
\item\label{3} For all $\x\in (\R^d)^p$ and $0\leq k\leq p$, there is a continuous linear map 
\[
\K_{\x}^k: W \longrightarrow V,\]
such that $\K_{\x}^0(W)=V_0$ and for all $F\in W$,
\bega 
\delta_{\x}\F &= \delta_{\x}\K_{\x}^k\F,\quad \forall 0\leq k\leq p
\\
  J_{x_k}\F &=J_{x_k} \K_{\x}^k\F,\quad \forall 1\leq k\leq p
\eega
\item\label{2} $\forall \F\in W$, the mapping $\x\mapsto \K_{\x}^k\F$ is continuous.
\end{enumerate}
We will call the family of maps $\K=\{\K_{\x}^k\colon \x\in (\R^d)^p,k=1,\dots,p\},$ a \emph{$p$-interpolator between $W$ and $V$}.
\miknew{We will say that $\K$ is a \emph{strong} $p$-interpolator, if in addition, the maps $\K_{\x}^k$ are surjective for all $1\le k \le p$.}
\end{definition}
\nnjump
\begin{definition}
We say that $V$ is a \emph{strong} $p$-interpolating space, if there exists a strong $p$-interpolator between $V$ and itself.
\end{definition}
\nnjump
The second assumption is what can be expected of an interpolator of $p$ points and the condition on the Jacobian determinant will be essential in the treatment of the Kac--Rice formula. The last assumption allows us to pass to the limit as a point $\x$ in $\U^d\setminus\Delta$ converges to $\Delta$.\jump

Given $F$ a $\CC^q$ Gaussian random field from $\U$ to $\R^d$, we define the $\CC^q$-support (or just support when the context is clear) of $F$ as the support of the associated Gaussian measure on $\CC^q(\U,\R^d)$ (see also \cite{NazarovSodin2016}). We prove the following general theorem concerning the finiteness of the $p$-th moment of the number of zeros of $F$ on any compact subset $K$ of $\U$.
\nnjump
\begin{theorem}
\label{thm:general}
Let $p,q,d$ be positive integers and $\U$ be an open subset of $\R^d$. Let $F:\U\rightarrow\R^d$ be a $\CC^q$ Gaussian random field with $\CC^q$-support $W$, and let $V$ be a \miknew{strong} \horny{} adapted to $W$, in the sense of Definition \ref{def:interpolator}. Then, Kac--Rice formula for $F$ holds and there exists a function $R_V:\U^p\setminus\Delta\rightarrow\R^+$ of class $\CC^q$, depending only on the space $V$, with the following properties:
\begin{itemize}
\item For any compact subset $K$ of $\U$, the function $R_V$ is integrable on $K^p\setminus\Delta$.
\item \miknew{There is a positive constant $C_F>0$ such that the Kac density $\rho_F$ satisfies
\[ \rho_\F\leq C_F R_V.\]}
\item \miknew{If the $p$-interpolator between $W$ and $V$ is strong, there is a positive constant $c_F>0$ such that the Kac density $\rho_F$ satisfies
\[c_F R_V\leq \rho_\F.\]}
\end{itemize}
In particular, for every compact subset $K$ of $\,\U$,
\[\E\left[\#Z(\F,K)^p\right] <+\infty.\]
\end{theorem}
\nnjump
We will show in section \ref{sec:proof} that Theorem \ref{thm:general} implies in particular Theorem \ref{thm1} and Theorem \ref{thm2}. 
\section{Multivariate interpolation}
The problem of interpolating a function on a open interval $\U$ of $\R$ by polynomials is \textit{well-poised}: given a positive integer $p$, there exists a \textbf{unique} polynomial of degree $p-1$ that interpolates a function $f$ at $p$ distinct points. We then say that the space of polynomials $\Pt_{p-1}^1(\R)$ is a \textit{Haar space} on $\U$. The situation in higher dimension is much more intricate.
\nnjump
\begin{theorem}[Haar--Mairhuber--Curtis]
Let $\U$ be an open subset of $\R^d$. There is no Haar space on $\U$ as soon as $d>1$.
\end{theorem}
\begin{proof}
See \cite[Thm. 2.4.1]{Dav75}.
\end{proof}
The problem of interpolating a function at $p$ distinct points is then always \textit{ill-poised} as soon as the ambient dimension is greater than one. In order to get around this non-existence theorem, there are two possible solutions whose choice will depend on the underlying goal.
\begin{itemize}
\item One can add some additional constraints on the relative position of the points. It has the advantage of keeping a space of interpolating functions that has the same dimension as the number of points, thus preserving uniqueness and reducing computations costs. This is generally the path taken in numerical problems.
\item One can drop the uniqueness and consider a larger space of functions for which the the problem of interpolating $p$ distinct points always has a solution. If the dimension of the space is greater than $p$ then the problem is underdetermined an has more than one solution. To recover uniqueness, one can add additional constraints on the interpolating function, such as the class of mean-value interpolation schemes. This approach is more adapted to theoretical problems and is the one we take in this paper.
\end{itemize}
We refer the reader to \cite{Nie12} and references therein for a thorough discussion on multivariate interpolation and various related problems.
\subsection{Kergin interpolation}\label{sec:kergin}
Let $p$ be a positive integer and $\U$ be an open convex subset of $\R^d$. In \cite{KERGIN}, the author provides an elegant solution to the interpolation problem at $p$ points of a function in $\CC^{p-1}(\U)$ by $d$-variate polynomials of total degree lower than $p-1$.
\nnjump
\begin{theorem}[Kergin]
Let $\x = (x_1,\ldots,x_p)$ be a collection of $p$ points in $\U$, with possible repetition. There is a unique mapping \[\Pi_{\x} : \CC^{p-1}(\U)\longrightarrow \Pt_{p-1}^d(\R)\]
such that
\begin{itemize}
\item $\Pi_{\x}$ is linear
\item For each $f\in\CC^{p-1}(\U)$, each integer $1\leq k\leq p$, each homogeneous partial differential operator with constant cefficients $q(D)$ of order $k-1$ and each $J\subset \{1,\ldots,p\}$ of size $k$, there exists a point $z=z(f,q,J)$ in the convex hull of $(x_j)_{j\in J}$ such that
\[q(D)\Pi_{\x}f(z) = q(D)f(z).\]
\end{itemize}
\end{theorem}
For $x\in\U$, The mapping $\Pi_{\x}$ enjoys many desirable properties that are similar to the classical Hermite--Lagrange interpolation in the one-dimensional case.
\nnjump
\begin{theorem}[Kergin Interpolation]
\label{thm02}\strut\\[-10pt]
\begin{enumerate}
\item The mapping $\Pi_{\x}$ is a continuous projector with respect to the the usual topology on $\CC^{p-1}(\U)$ of compact convergence of the derivatives up to order $p-1$.
\item For each $f\in \CC^{p-1}(\U)$, the mapping $\x\rightarrow \Pi_{\x}f$ from $(\R^d)^p$ to $\R$ is continuous. 
\item If the multiplicity of $x_i$ in $\x$ is $n$ then
\[\forall \,|\alpha|< n,\quad \partial_\alpha f(x_i) = \partial_\alpha \Pi_{\x}f(x_i).\]
\end{enumerate}
\end{theorem}

In \cite{MICCHELLI}, the authors provide an explicit formula for the Kergin interpolant that generalizes the Hermite-Gennocchi (see \cite{MICCHELLI}) formula in the one-dimensional case. We define the standard simplex of dimension $p-1$ as  
\[\Sigma^{p-1} = \enstq{(v_1,\ldots,v_p)\in (\R_+)^p}{\sum_{i=1}^pv_i = 1},\]
and for any vector  of $p$ distinct points $\x=(x_1,\dots,x_p)\in (\R^d)^p,$ we define the linear functional $\int_{[\x]}$ such that
\[\int_{[\x]} f = \int_{\Sigma^p}f(v_1x_1+\ldots+v_px_p)\dd v_1\ldots\dd v_p.\]
For a positive integer $r$, and $x\in\U$ we denote by $D^r_xf$ the $r^{th}$ derivative of $f$ at point $x$, meant as a symmetric multilinear function $D^r_xf: \R^d\times \ldots \times \R^d\to \R$ defined by
\[
D^r_xf :(u_1,\ldots,u_r)\mapsto\sum_{a_1,\ldots,a_r=1}^d\partial_{a_1}\ldots\partial_{a_r}f(x)\,u_1^{a_1}\ldots  u_r^{a_r}.
\]
Here, $u^a$ denotes the $a^{th}$ component of the vector $u=(u^1,\ldots,u^d)^T\in \R^d.$ Notice that $D^r_xf (u_1,\ldots,u_r)$ is a polynomial of degree $r$ in the variables $u_1,\dots, u_r$.
\nnjump
\begin{theorem}\label{thm:micchelli}
For each $f\in \CC^{p-1}(\U)$, $\x\in\U^p$ and $z\in \U$,
\[
\left(\Pi_{\x} f\right)(z)=\sum_{r=0}^{p-1} \int_{[(x_1,\ldots,x_{r+1})]}  D_v^{r}f \left((z-x_1),\ldots ,(z-x_{r})\right) dv,
\]
\end{theorem}
\begin{proof}
See \cite{MICCHELLI}.
\end{proof}
For $k\in\{1,\ldots p\}$ we can consider the projector $\Pi_{\x}^k :=\Pi_{(\x,x_k)}$ 
\[\Pi_{\x}^k : \CC^{p}(\U)\longrightarrow \Pt_p^d(\R),\]
associated with the collection of points $(x_1,\ldots,x_p,x_k).$ Then for all $k,l\in \{1,\ldots,p\}$ one has
\[\Pi_{\x}^kf(x_l) = f(x_l)\quand \nabla \Pi_{\x}^kf(x_k) = \nabla f(x_k),\]
from of point $3$ of Theorem \ref{thm02}. By applying the Kergin interpolation on each component, we can define Kergin interpolation for vector fields
\be \label{eq:dKerg}
(\Pi^k_{\x})^{\otimes d}:\CC^p(\U,\R^d)\longrightarrow (\Pt^d_p(\R))^d, \quad 
(\Pi^k_{\x})^{\otimes d}\begin{pmatrix}f^1 \\ \vdots \\f^d \end{pmatrix}=\begin{pmatrix}\Pi^k_{\x}f^1 \\\vdots \\ \Pi^k_{\x}f^d\end{pmatrix}.
\ee
By identifying $\C^d$ with $\R^{2d}$ we also define in a similar fashion the Kergin projectors $\Pi_{\x}^k$ and $(\Pi_{\x}^k)^{\otimes d}$ for $\x\in\C^d$, that act on holomorphic functions and holomorphic vector fields, respectively. Now  let $W$ be a subspace of $\CC^{p+1}(\U)$. We define
\[\nabla W = \enstq{F\in \CC^p(\U,\R^d)}{\exists f\;\;\text{s.t}\;\;\nabla f = F},\]
the space of gradient fields of class $\CC^p$. By the Poincar\'{e}'s Theorem, this is equivalent to say that elements in $W$ satisfy the usual Schwarz identities for gradient fields, see \cite{lee2003introduction}. We still denote by $\nabla$ the complex gradient and we define $\nabla W$ similarly when $W$ is a subset of the space of holomorphic functions on an open subset $\U$ of $\C^d$.
\nnjump
\begin{lemma}\label{lem:gradientKergin}
Let $\U$ be an open convex subset of $\R^d$. For $\x\in \U^p$ and $1\leq k\leq p$, the operator $(\Pi^k_{\x})^{\otimes d}$ is a projector from the space $\nabla \CC^{p+1}(\U)$ onto the polynomial space $\nabla\Pt_{p+1}^d(\R)$.
\end{lemma}
\begin{proof}
Since, the Kergin interpolation is a projector onto the space of polynomials, it is enough to show that the image of $\nabla \CC^{p+1}(\U)$ by the projector $(\Pi^k_{\x})^{\otimes d}$ in contained in $\CC^{p+1}(\U)$. By the Poincar\'{e}'s theorem, it suffices to show that for a function $f$ in $\CC^{p+1}(\R^d)$, a collection $\x=(x_0,\ldots,x_p)$ of $p+1$ points in $\U$ and $i,j\in\{1,\ldots,d\},$ one has
\be\label{eq:irrotational}
\de_j\tyu\Pi_{\x}\de_if\uyt=\de_i\tyu\Pi_{\x}\de_jf\uyt.
\ee
To prove Equation \eqref{eq:irrotational}, we will use the explicit formula of $\Pi_{\x},$ denoting the simplex $[x_0,\ldots,x_{r}]$ as $[\x]_r$.  For any $z=(z^1,\ldots,z^d)^T\in \R^d,$ we have
\bega
\de_j &\tyu\Pi_{\x}\de_if\uyt (z)
=
\de_j\sum_{r=0}^{p} \int_{[\x]_r}\sum_{a_1,\ldots,a_r=1}^d\partial_{a_1}\ldots\partial_{a_r}\partial_i f(v)(z^{a_1}-x_1^{a_1})\ldots  (z^{a_r}-x_r^{a_r}) dv
\\
&=
\sum_{r=0}^{p} \int_{[\x]_r}\sum_{a_1,\ldots,a_r=1}^d\partial_{a_1}\ldots\partial_{a_r}\partial_if(v)\sum_{\ell=1}^r\delta(j,a_\ell)\frac{(z^{a_1}-x_1^{a_1})\ldots  (z^{a_r}-x_r^{a_r})}{(z^{j}-x_j^{j})} dv
\\
&=\sum_{\ell=1}^rS_\ell(i,j,z),
\eega
where $\delta(j,a)=0$ if $j\neq a$ and $\delta(j,j)=1.$ Each term $S_\ell(i,j,z)$ in the latter expression is symmetric in $i,j$. We show this for $\ell=1,$ since the other terms are analogous.
\bega
S_1(i,j,z)&=\sum_{r=0}^{p} \int_{[\x]_r}\sum_{a_1,\ldots,a_r=1}^d\partial_{a_1}\ldots\partial_{a_r}\partial_if(v)\delta(j,a_1)\frac{(z^{a_1}-x_1^{a_1})\ldots  (z^{a_r}-x_r^{a_r})}{(z^{j}-x_j^{j})} dv
\\
&=
\sum_{r=0}^{p} \int_{[\x]_r}\sum_{a_2,\ldots,a_r=1}^d\partial_j\partial_{a_2}\ldots\partial_{a_r}\partial_if(x)(z^{a_2}-x_2^{a_2})\ldots  (z^{a_r}-x_r^{a_r}) dv
\\
&=S_1(j,i,z)).
\eega
\end{proof}
\begin{lemma}\label{lem:holomorphicKergin}
$\U$ be an open convex subset of $\C^d$. For $\x\in \U^p$ and $1\leq k\leq p$, the operator $(\Pi^k_{\x})^{\otimes d}$ is a projector from the space of holomorphic functions on $\U$ onto the polynomial space $\Pt_p^d(\C)$.
\end{lemma}
\begin{proof}
The proof is similar to Lemma \ref{lem:gradientKergin}, but simpler. It is enough to observe that if $f$ is holomorphic, then the integrand in Theorem \ref{thm:micchelli} is a complex polynomial in $x.$ See also \cite{AnderssonPassare}.
\end{proof}
\subsection{Examples of interpolating spaces}
\subsubsection{Bezout's theorem}
\label{sec.bezout}
 We first recall the Bezout's theorem concerning the number of common roots of a system of polynomial equations. This is a central point in our proof of the finiteness of the moments number of zeros. The space of polynomials fields is sufficiently large in order to interpolate a smooth function, while admitting an a.s. universal bound for number of zeros of its elements.\nnjump
\begin{theorem}[Bezout]
\label{thm:bezout}
Let $P= (P_1,\ldots,P_d) \in (\Pt_p^d(\C))^d$ be a collection of $d$ polynomials such that $0$ is a regular value for $P$. Then
\[\#Z(P,\C^d)\leq \prod_{i=1}^d\deg (P_i).\]
\end{theorem}
\begin{proof}
See \cite[Lemma 11.5.1]{BCR}.
\end{proof}
\begin{corollary}
\label{cor:bezout}
Let $V$ be either the space $(\Pt_p^d(\R))^d, (\Pt_p^d(\C))^d, \nabla \Pt_{p+1}^d(\R)$ or $\nabla \Pt_{p+1}^d(\C)$. Then for almost every polynomial $P$ vector field in $V$, 
\[\#Z(P,\C^d)\leq p^d.\]
\end{corollary}
\begin{proof}

Observe that in each of these four cases, the function 
\begin{align*}
\phi:V\times \C^d&\longrightarrow \C^d\\
(P,x)&\longmapsto P(x)
\end{align*}
is a smooth submersion. By the Parametric Transversality theorem, see \cite[Chp. 3]{Hirsch}, this implies that the subset of $V$ consisting of the polynomials for which $0$ is a not a regular value has zero measure. The conclusion then follows from Bezout's Theorem \ref{thm:bezout}.
\end{proof}

The Kergin interpolation introduced in the previous section yields several examples of adapted $p$-interpolating spaces in the sense of Definition \ref{def:interpolator}.
\nnjump
\begin{lemma}\label{lem:Vpgood}
Let $p$ be a positive integer and $\U$ be a open subset of $\R^d$. Then the space $V=(\Pt_p^d(\R))^d$ of real polynomial vector fields is a \miknew{strong} \sexy{} \miknew{with $V_0=(\Pt_{p-1}^d(\R))^d$,} adapted to the space $W=\CC^p(\U,\R^d)$.
\end{lemma}
\begin{proof}
We first show that the space $V=(\Pt_p^d(\R))^d$ is a \sexy{} with \miknew{$V_0=(\Pt_{p-1}^d(\R))^d$.} Let $\y=(y_1,\dots,y_p)$ be a collection of $p$ distinct points in $\R^d$, $v_1,\ldots,v_p$ be a collection of $p$ vectors in $\R^d$ and $\alpha\in\R$. 
One can explicit a smooth function $F\in\CC^{\infty}(\R^d,\R^d)$ such that for all $1\leq l \leq p$,
\[F(y_i) = z_i\quand J_{y_k}F = \alpha.\]
Then Kergin interpolation and Theorem \ref{thm02} implies that the polynomial $(\Pi_{\y}^k)^{\otimes d}F$ satisfies
\[(\Pi_{\y}^k)^{\otimes d}F(y_i) = z_i\quand J_{y_k}(\Pi_{\y}^k)^{\otimes}F = \alpha.\]
\miknew{Moreover, the first condition can be realized by the $p$-point interpolator $\K^0_{\y}:=(\Pi_{\y})^{\otimes d}$ whose image is $V_0$.}
In particular, Properties A and B in Definition \ref{def:interpolating} are satisfied and Property C is a consequence of Corollary \ref{cor:bezout}. The properties of Kergin interpolation in Theorem \ref{thm02} also directly imply that $\K_{\x}^k=(\Pi_{\y}^k)^{\otimes d}$ is a $p$-interpolator and, thus, that the space $(\Pt_p^d(\R))^d$ is a \sexy{} adapted to the space $\CC^p(\U,\R^d)$. \miknew{To see that it is strong, see Remark \ref{rem:extKergin}.}
\end{proof}
The proof of the subsequent three lemmas are in all points similar, with the additional remarks that the Kergin interpolating polynomial of a gradient field is a gradient polynomial field, according to Lemma \ref{lem:gradientKergin}, and that the Kergin interpolating polynomial of an holomorphic field is an holomorphic polynomial field, according to Lemma \ref{lem:holomorphicKergin}.
\nnjump
\begin{lemma}\label{lem:Vpgood2}
Let $p$ be a positive integer and $\U$ be an open subset of $\R^d$. Then the space $V=\nabla \Pt_{p+1}^d(\R)$ of gradients of real polynomials is a \miknew{strong} \sexy{} \miknew{with $V_0=\nabla \Pt_{p}^d(\R)$,} adapted to the space $W=\nabla \CC^{p+1}(\R^d)$ (as a subspace of $\CC^p(\U,\R^d)$).
\end{lemma}
\begin{proof}
We rehearse the arguments of the previous Lemma \ref{lem:Vpgood}, taking into account that Kergin sends gradient fields to gradient polynomials, according to Lemma \ref{lem:gradientKergin}.
\end{proof}
\begin{lemma}\label{lem:Vpgood3}
Let $p$ be a positive integer and $\U$ be a open subset of $\C^d$. Then the space $V=(\Pt_p^d(\C))^d$ of all holomorphic polynomial vector fields is a \miknew{strong} \sexy{} \miknew{with $V_0= \Pt_{p-1}^d(\C)^d$,} adapted to the space $W=\HH(U,\C^d)$ (as a subspace of $\CC^p(\U,\R^{2d})$).
\end{lemma}
\begin{proof}
By Lemma \ref{lem:holomorphicKergin}, we can rehearse the argument the proof of Lemma \ref{lem:Vpgood}, using the complex Kergin projector as $p$-interpolator.
\end{proof}
\begin{lemma}\label{lem:Vpgood4}
Let $p$ be a positive integer and $\U$ be a open subset of $\C^d$. Then the space $V=\nabla \Pt_{p+1}^d(\C)$ of gradients of complex polynomials is a \miknew{strong} \sexy{} \miknew{with $V_0=\nabla \Pt_{p}^d(\C)$,} adapted to the subspace $W=\nabla \HH(U)$ (as a subspace of $\CC^p(\U,\R^{2d})$).
\end{lemma}
\begin{proof}
We can argue as in the proof of Lemma \ref{lem:Vpgood}, using both Lemma \ref{lem:gradientKergin} and Lemma \ref{lem:holomorphicKergin}.
\end{proof}
%
\begin{remark}
\label{rem:extKergin}
The proof of Lemma \ref{lem:Vpgood} also shows that the space \miknew{$V=(\Pt_p^d(\R))^d$ is a \sexy{}, with $V_0=(\Pt_{p-1}^d(\R))^d$, adapted to a closed subset $W$ of $\CC^p(\U,\R^d)$  as soon as for all $\x\in \U^p$ the mappings
\bega 
\Pi_{\x} &: W\mapsto (\Pt_{p-1}^d(\R))^d 
\eega
is surjective. In all these examples, the interpolator is strong, which happens as soon as also the maps 
\bega 
\Pi_{\x}^k &: W
\mapsto (\Pt_p^d(\R))^d
\eega
are surjective. 
Clearly, this is true for $W=(\Pt_p^d(\R))^d$. Therefore, $(\Pt_p^d(\R))^d$ is in fact a strong $p$-interpolating space.}
A similar statement holds for the other three variants.
\end{remark}
\section{Proof}\label{sec:proof}
In this section we prove the general Theorem \ref{thm:general}. As explained in the introduction, the proof relies on the Gram--Schmidt orthogonalization procedure on the family of evaluation maps at points $y_1,\ldots,y_p\in \U^p\setminus\Delta$. Before making a few observations around Kac--Rice formula and \horny{}s, we recall the very classical Gram-Schmidt theorem.

Let $V$ be a finite dimensional vector space endowed with a scalar product $\langle\,,\,\rangle$.

%
For a linear subspace $F\subset V$ we define $\Proj_F$ as the orthogonal projector on the space $F$.
\nnjump
\begin{theorem}[Gramm--Schmidt]
\label{thm01}
Let $V$ be a euclidean space and let $v_1,\ldots,v_p$ be a free family of vectors in $V$. Let $V_i = \Span(v_1,\ldots,v_i)$. Then there is a unique orthogonal family of $V$, denoted $u_1,\ldots,u_p$ such that
\begin{itemize}
\item $\forall 1\leq i\leq p, \quad u_i\in V_i$
\item $\langle u_i,u_j\rangle = \delta_{ij}$
\item $\langle u_i, v_i\rangle >0$.
\end{itemize}
Explicitly, one has
\be \label{eq:GS}
u_i = \frac{v_i - \Proj_{V_{i-1}}(v_i)}{\|v_i - \Proj_{V_{i-1}}(v_i)\|}, \quad \text{where } V_0=\{0\}.\ee
\end{theorem}
\begin{remark}\label{rem:smoothGS}
Let $\phi(p,V)$ be the open subset of $V^{p}$ consisting of all free families of $p$ elements in $V.$ Notice that the theorem implies the existence of a function $A\colon \phi(p,V)\to GL(p)$ such that $A(v_1,\dots,v_p)$ is the matrix of change of basis from the basis $(v_1,\dots,v_p)$ to the orthonormal basis $(u_1,\dots,u_p).$ One can easily deduce from the formula \eqref{eq:GS} that $A$ is a smooth function.
\end{remark}
\subsection{Kac--Rice and \horny{}s}
\label{rem:transversality}
The Kac-Rice formula cited in Theorem \ref{thm03} is more general and holds also for a certain class non-Gaussian fields, see \cite{Aza09,AdlerTaylor}, and also \cite{MathiStec, KRStec} for further generalizations. It is plausible that the proof in this paper generalizes to a certain class of non-Gaussian fields. The Gaussian assumption of this paper is mainly present to ensure the validity of Kac--Rice formula and the non-degeneracy of the random field, but nowhere in the proof do we use the explicit Gaussian density of the field throughout computations.
\jump

In \cite[Thm. 6.3]{Aza09}, the validity of Kac--Rice formula holds on the additional assumption that, almost surely, $0$ is a regular value for $F$, i.e. for every $x\in F^{-1}(0),$ the differential $\nabla_x F$ is surjective. It is again proved in \cite[Prop. 6.5]{Aza09} that this technical assumption is true when the random field $F$ is of class $\CC^2$ and has non-degenerate density at each point $x\in\U$. Analogously, it has been proved in \cite[Theorem 7]{dtgrf} that when $F$ is a smooth Gaussian field having a non-degenerate density at each point $x\in\U$, then $0$ is a.s. a regular value. When $F$ is a Gaussian field from a subset $\U$ of $\R^d$ to $\R^d,$ the latter proof reduces to an application of  Sard's theorem for $\CC^1(\R^d,\R^d)$ functions.\footnote{In general, Sard's theorem holds for functions $\F\in\CC^k(\R^d,\R^{d'})$ such that $k\ge \max\{1,d-d'+1\}$, see \cite{lee2003introduction}.} It implies in this case the validity of Kac--Rice formula for $\CC^1$ Gaussian random fields with non-degenerate evaluations.
\jump

We first prove the validity of Kac--Rice formula in the context of \horny{}s. \miknew{We denote by $\Ss^+(V_0)$ the space of semipositive definite quadratic forms on $V_0^*$. The covariance tensor of a Gaussian random vector in $V_0$ lives naturally in this space.}

\nnjump
\begin{lemma}
\label{lemma-conti}
Let $q,d,p$ be positive integers. Let $\F\colon \U\to \R^d$ be a $\CC^q$ Gaussian field whose support $W$ admits an adapted \horny{} $V$.  Then, for all $\x\in \U^p$, the random Gaussian function $\K_{\x}^0 \F$ has a non-degenerate density on $V_0$ and the mapping
\begin{align*}
\U^p&\longrightarrow \Ss^+(V_0)\\
\x&\longmapsto \Cov(\K_{\x}^0 \F)
\end{align*}
is continuous.
\end{lemma}
\begin{proof}
The proof directly follows from the definition of a \horny{}, since the mapping $\K_{\x}^k\F$ is continuous and surjective.
\end{proof}

\begin{lemma}\label{lem:hyp}
Let $q,d,p$ be positive integers. Let $\F\colon \U\to \R^d$ be a $\CC^q$ Gaussian field whose $\CC^q$-support $W$  admits an adapted \horny{} $V$. Then, the asumptions of Theorem \ref{thm03} are satisfied and the Kac-Rice formula \eqref{eq:kr} for the $p$-th moment holds.
\end{lemma}
\begin{proof}
It follows directly from the previous Lemma \ref{lemma-conti}, since
\[(F(y_1),\ldots,F(y_p)) = (\K_{\y}^0F(y_1),\ldots,K_{\y}^0F(y_p))\]
is a non-degenerate Gaussian vector on $V$.
\end{proof}
\subsection{Proof of the general theorem: Theorem \ref{thm:general}}
We can now prove Theorem \ref{thm:general}. Let $p,q,d$ be positive integers and $\U$ be an open convex subset of $\R^d$. Let $F:\U\rightarrow\R^d$ be a $\CC^q$ Gaussian random field with $\CC^q$-support $W$, and let $V$ be a \horny{} adapted to $W$, in the sense of Definition \ref{def:interpolator}. We endow the space \miknew{$V_0^*$ with a scalar product $\langle\,.\,\rangle$ and we endow the space of polynomials $\Pt_d[V]$ with a norm $\|.\|$.}\jump
 
For an element $\y\in \U^p\setminus \Delta$, we define $D_{\y}$ as the family of $dp$ linear forms on $\CC^0(\U,\R^d)$ obtained by applying the Gram--Schmidt procedure described in Theorem \ref{thm01}, on the family $\delta_{\y}$ with respect to the scalar product defined on \miknew{$V_0^*$}. Letting $A_{\y}\in GL(dp)$ be the associated transformation matrix, if we interpret $\delta_{\y}$ and $D_{\y}$ as linear functions from $\CC^0(\U,\R^d)$ to  $(\R^d)^p $,
then
\[\delta_{\y} = A_{\y}\underline{D}_{\y}. 
\]
\begin{proposition}\label{prop:keyD}
The linear functions $D_{\y}$ have the following two key properties. For every $\F\in W,$
\be 
\delta_{\y}\F=0 \iff  D_{\y}\F=0  
\quad \text{and} \quad \miknew{D_{\y}\big|_{V_0} \text{ is orthonormal in $V_0^*$.}}
\ee 
\end{proposition}
\begin{proof}
The claimed property is true by construction and, in particular, thanks to Assumptions \ref{3} and \ref{a}.
\end{proof}
We denote by $\Proj_{\Ker(\delta_{\y})}$ the orthogonal projector on $V$ onto the subspace $\Ker(\delta_{\y})\cap V$. For $\y\in\U^p\setminus\Delta$ and $1\leq k\leq p$ we let
\[\lambda_{\y}^k = \|J_{y_k}\circ \Proj_{\Ker(\delta_{\y})}\|,\]
the norm $\|\cdot \|$ is the one defined in the beginning of this section on the space $\Pt_d[V]$. By Assumption \ref{b}, the polynomial
$J_{y_k}$ is non-zero on the subset $\Proj_{\Ker(\delta_{\y})}$, which implies that the quantity $\lambda_{\y}^k$ is positive.  We can thus define the normalized polynomial in $\Pt_d[V]$
\[
h^k_{\y} = \frac{J_{y_k}\circ \Proj_{\Ker(\delta_{\x})}}{\lambda_{\y}^k},
\]
and the polynomial in $\Pt_d[W]$
\[H_{\y}^k:=h^k_{\y}\circ \K_{\y}^k.\]
\miknew{
\begin{remark}\label{rem:Hkzero}
While $h_{\y}^k$ is always a nonzero polynomial on $V$, it is possible, in general, that $H^k_{\y}=0$ as a polynomial on $W$. If the $p$-interpolator $\K$ is strong, it follows from Property \ref{3} and the surjectivity of $\K^k_{\y}$ that the polynomial $H^k_{\y}$ is nonzero for all $1\le k \le p$. By construction, in this  case, also the restricton $H^k_{\y}|_{W\cap \ker(\delta_{\x})}$ is nonzero.
\end{remark}
}
\nnjump
\begin{proposition}\label{prop:keyH}
The polynomial functions $H^k_{\y}$ have the following two key properties. For every $\F\in W,$ such that $\delta_{\y}\F=0,$ we have
\be 
J_{y_k}\F= \lambda^k_{\y} H^k_{\y}(\F) \quad \text{and} \quad h^k_{\y}\big|_{V} \text{ is in the sphere of $\Pt_d[V]$.}
\ee 
\end{proposition}
\begin{proof}
The relations in this statement are true by construction, in particular, thanks to Assumptions \ref{3} and \ref{b}.
\end{proof}
Let us define two functions $R,\sigma_\F\colon \U^p\setminus \Delta\to [0+\infty)$ such that 
\be \label{eq:R}
R(\y) := \frac{\prod_{k=1}^p|\lambda^k_{\y}|}{|\det A_{\y}|}\quand \sigma_\F(\y) = \E\cdbr{\prod_{k=1}^p |H_{\y}^{k}\F|}{D_{\y}\F = 0}\psi_{D_{\y}\F}(\underline{0}).\ee
Note that the function $R$ does not depend on the process $\F$ but only on the space $V$. Moreover, recalling Remark \ref{rem:smoothGS}, we observe that the function
\be 
\y\mapsto\delta_{\y} \in \phi(dp,V_0^*)=
\kop \text{free families of $pd$ elements in \miknew{$V_0^*$}} \pok
\ee
 is of class $\CC^q,$ whenever $V\subset \CC^q(\U,\R^d).$ It follows that the matrix valued function $\y\mapsto A_{\y}\in GL(dp)$ produced by the Gram-Schmidt's theorem is also of class $\CC^q$. For the same reason, the function $\y\mapsto \lambda_{\y}^k$ and $R$ are of class $\CC^q.$
\begin{proposition}\label{prop:krR}
The Kac--Rice density 
\[\rho_\F:\y\mapsto\E\cdbr{\prod_{k=1}^p |J_{y_k}\F|}{\delta_{\y}\F= 0}\psi_{\delta_{\y}\F}(\underline{0})\]
can be rewritten as
\[\rho_\F(\y) = R(\y)\sigma_\F(\y),\]
where $R$ and $\sigma_\F$ are the functions defined in \eqref{eq:R}. 
\end{proposition}
\begin{proof}
One has
\bega 
\rho(\y) 
&=\E\cdbr{\prod_{k=1}^p |J_{y_k}\F|}{A_{\y}D_{\y}\F = 0}\psi_{A_{\y}D_{\y}\F}(\underline{0})
\\
&=\E\cdbr{\prod_{k=1}^p |J_{y_k}\F|}{D_{\y}\F = 0}\psi_{D_{\y}\F}(\underline{0})\frac{1}{|\det A_{\y}|}
\\
&=\E\cdbr{\prod_{k=1}^p |H_{y_k}\F|}{D_{\y}\F = 0}\psi_{D_{\y}\F}(\underline{0})\frac{\prod_{k=1}^p|\lambda_{\y}^k|}{|\det A_{\y}|}.
\eega
\end{proof}

\begin{lemma}\label{lem:subsequence}
For every sequence $(\y_m)_{m\geq 0}$ of points in $\U^p\setminus \Delta$ that converges in $\U^p$ to a limit point $\x$, one can extract a subsequence $(\y_{\phi(m)})_{m\geq 0}$ such that 
\begin{itemize}
\item The sequence of free families $(D_{\y_{\phi(m)}})_{m\geq 0}$ on $W^*$ converges pointwise towards a limit free family $D$ on $W^*$.
\item For $1\leq k\leq p$, the sequence of polynomials $(H_{\y_{\phi(m)}}^k)_{m\geq 0}$ defined on $W$ converges pointwise towards a limit polynomial $H^k$ on $W$. \miknew{If the $p$-interpolator $K$ is strong, then $H^k\neq 0$.}
\end{itemize}
\end{lemma}
\begin{proof}
The result can be deduced by Propositions \ref{prop:keyD} and \ref{prop:keyH} as follows. Let $(\y_m)_{m\geq 0}$ be a sequence of points in $\U^p\setminus \Delta$ that converges in $\U^p$ to a limit point $\x$. The space $V$ is finite dimensional, and so are the spaces $V^*$ and $\Pt_d[V]$. One can then find by compactness a subsequence $(\y_{\phi(m)})_{m\geq 0}$ such that 
\begin{itemize}
\item The sequence of orthogonal families $(D_{\y_{\phi(m)}})_{m\geq 0}$ in $V^*$ converges towards a limit orthogonal family $D$ in $V^*$.
\item For $1\leq k\leq p$, the sequence of unit norm polynomials $(h_{\y_{\phi(m)}}^k)_{m\geq 0}$ defined on $V$ converges towards a limit unit norm polynomial $h^k$ on $V$.
\end{itemize}
Now for $1\leq k\leq p$, one has the identities
\[D_{\y_m} = D_{\y_m}\circ \K_{\y_m}^k\quand H_{\y_m}^k=h_{\y_m}^k\circ \K_{\y_m}^k,\]
as functions on $W$. We can extend the mapping $D$ and the polynomial $H_k$ for $1\leq k\leq p$ defined as function on $V$, to the whole space $W$ via the identity
\[D=D\circ \K_{\x}^k\quand H^k:=h^k\circ \K_{\x}^k.\]
Since the mapping $\x\mapsto \K_{\x}^k$ is continuous (Assumption \ref{2}), we deduce the following pointwise convergence of functions on $W$:
\[\lim_{m\rightarrow+\infty} D_{\x_{\phi(m)}} = D\circ \K_{\x}^k=D\quand \lim_{m\rightarrow+\infty} H_{\x_{\phi(m)}}^k = h^k\circ\K_{\x}^k=H^k.\]
Since the family $D$ is free as a family in $V^*$ it is also free in $W^*.$ Similarly, the polynomial $h^k$ is non-zero on $V\cap\Ker(D)$ and thus\miknew{, when $\K^k_{\x}$ is surjective,} $H^k$ is non-zero on $W\cap\Ker(D)$.
\end{proof}

\begin{lemma}\label{lem:bound}
Let $K$ be a compact subset of $\,\U$. There are 
 constants $c_\F,C_\F$, depending only on the compact $K$ and on the distribution of the underlying process $\F$, such that for all $\y\in K^p\setminus\Delta$, one has
\[c_\F\leq\sigma_\F(\x)\leq C_\F.\]
\miknew{If the interpolator is strong, then $c_F>0$. }
\end{lemma}
\begin{proof}
It suffices to show that for every sequence $(\y_m)_{m\geq 0}$ in $K^p\setminus\Delta$, there is a subsequence $(\y_{\phi(m)})_{m\geq 0}$ and constants $c,C$ (that may depend on the subsequence) such that
\[\forall m\geq 0,\quad c\leq \sigma(\y_{\phi(m)})\leq C.\]
Let $(\y_m)_{m\geq 0}$ be such a sequence. We take a converging subsequence towards a limit point $\x\in K^p$ such as in Lemma \ref{lem:subsequence} (we still denote this subsequence as $(\y_m)_{m\geq 0},$ for simplicity). Let $D$ and $(H^k)_{1\leq k\leq p}$ be the quantities defined as limit quantities in Lemma \ref{lem:subsequence}. We define 
\[\sigma_\F = \E\cdbr{\prod_{k=1}^p |H^k\F|}{D\F = 0}\psi_{D\F}(\underline{0}).\]
The quantity $\sigma_\F$ is well-defined since $D\F$ is a non degenerate Gaussian vector. 
From the convergence in Lemma \ref{lem:subsequence}, one has 
\[\lim_{m\rightarrow+\infty} \sigma_\F(\x_m) = \sigma_\F\in \R.\]
\miknew{The existence of the limit ensures the upper bound.
Finally, if the $p$-interpolator is strong, then for all $1\leq k\leq p$, the polynomial $H^k$ is nonzero. By construction, this is equivalent to say that $H^k$ is nonzero on the space $W\cap\Ker(D)$ (see Remark \ref{rem:Hkzero}), therefore the limit $\sigma_F$ is bounded away from zero. Thus, in this case, the sequence is eventually bonded below by a constant $c>0$.
}
\end{proof}
\begin{lemma}\label{lem:integrable}
For every compact subset $K\subset \R^d,$ the function $R(\x)$ is integrable on $K^p$ and 
\[\E\qwe\#Z(\F,K)^{[p]}\ewq<+\infty.\]
\end{lemma}
\begin{proof}
Let $G$ be a non-degenerate Gaussian process on $V$. \miknew{By assumption $V$ is a \horny{}, with subspace $V_0$, that admits a strong $p$-interpolator adapted to itself}, thus we can apply the previous results to the random field $G$, and for $W=V$. In particular,  since the function $R,$ defined in \eqref{eq:R} depends only on $V$ \miknew{ and $V_0$,} we have that the Kac-Rice density for $G$ can be written as $\rho_{G}=R\cdot \sigma_G$ and that $\sigma_G|_{K^p\setminus \Delta}$ is bounded below by a constant $c_G>0,$ by Lemma \ref{lem:bound} \miknew{applied to the strong $p$-interpolator between $V$ and $(V,V_0)$}. Then,
\begin{align*}
\E\qwe\#Z(\F,K)^{[p]}\ewq&= \int_{K^p} R(\x)\sigma_\F(\x)\dd \x\\
&\leq \frac{C_\F}{c_G}\int_{K^p} R(\x)\sigma_G(\x)\dd \x\\
&\leq \frac{C_\F}{c_G}\E\qwe\# Z(G,K)^{[p]}\ewq,
\end{align*}
\miknew{where the constant $C_F>0$ is also provided by Lemma \ref{lem:bound}, applied to the (not necessarily strong) interpolator between $W$ and $(V,V_0)$.} According to the definition of a \sexy{}, for almost every realization of $G$, one has that
\[\#Z(\F,K)<C_K.\]
Then
\[\E\qwe\# Z(\F,K)^{[p]}\ewq \leq \frac{C_\F}{c_G}C_K^{[p]},\]
and the conclusion follows.
\end{proof}
Recollecting together all the results of this subsection, we obtain a proof of Theorem \ref{thm:general}. 

\subsection{Proof of Theorem \ref{thm1} and Theorem \ref{thm2}}
\begin{proof}
We start with the proof of Theorem \ref{thm1}. It is sufficient to prove the theorem when $K$ is an arbitrarily small compact neighborhood $\BB$ of $x$ in $\U,$ for every $x\in \U.$ Let $x\in \U$ and consider the \miknew{$p$-tuple $\x_0=(x,\dots,x)\in\U^{p}$. Then, the Kergin interpolator at $\x_0$ coincides with the Taylor polynomial of order $p-1$ at $x,$ by Theorem \ref{thm02}. It follows that the Gaussian vector
\be 
\Pi_{\x_0}^{\otimes d}F=T_{x}^{p-1}F
\ee
is non-degenerate. By continuity (Assumption \ref{2}), the same holds for all $\x\in \BB^{p},$ for some small enough $\BB$ convex neighborhood of $x$ such that $\BB\subset \U$.}\jump

Let us consider the restricted random field $F|_\BB\colon \BB\to \R^d$ and let $W\subset \CC^p(\BB,\R^d)$ be its support. We take $V = (\Pt_p^d(\R))^d$ and for $\x\in \BB^p$ and $1\leq k\leq p$, we define 
\be\label{eq:kergWV} 
\K_{\x}^k = (\Pi_{\x}^k)^{\otimes d}:W\to V,\ee
the Kergin interpolant at the $(p+1)-$tuple $(\x,x_k)$ on each coordinates. \miknew{Moreover, we take $V_0=(\Pt_{p-1}^d(\R))^d$ and define 
\be 
\K_{\x}^0 = (\Pi_{\x})^{\otimes d}:W\to V_0,
\ee
the Kergin interpolant at the $p$-tuple $\x$.
These maps are well defined on $W$ because $\BB$ is convex, see Theorem \ref{thm:micchelli}. The function $\K_{\x}^0$ is surjective because the random vector $\K_{\x}^0F$ is non-degenerate. 
The mapping $\K_{\x}^k$ satisfies Assumptions \ref{3} and \ref{2} of a $p$-interpolator, according to Theorem \ref{thm02}, and it satisfies Assumption \ref{4} because of  Lemma \ref{lem:Vpgood} above. Thus, $V$ is a \horny{} adapted to $W$.}\jump

To conclude, recall that there is a universal constant $C_p>0.$ such that for every positive random variable $\alpha,$ one has $\E[\alpha^{p}]\le C_p(1+ \E[\alpha^{[p]}]),$ hence the conclusion follows from Theorem \ref{thm:general}, which yields
\be 
\E\qwe\# Z(\F,\BB)^p\ewq\le C_p\tyu 1+\E\qwe \#Z(\F,\BB)^{[p]}\ewq\uyt<+\infty. 
\ee
As for the proof of Theorem \ref{thm2}, we can rehearse the same arguments, but this time using using Lemma \ref{lem:Vpgood2} \miknew{with $V=\nabla \Pt_{p+1}^d(\R)$ and $V_0=\nabla \Pt_{p}^d(\R)$. Similarly, for the complex case we use complex Kergin interpolation with interpolating space $V = (\Pt_{p}^d(\C))^d$ or $V=\nabla \Pt_{p+1}^d(\C)$ and Lemma \ref{lem:Vpgood3} and  \ref{lem:Vpgood4}.
}\jump


Now let $(F_n)_{n\geq 0}$ be a sequence of random field converging in distribution, for the $\CC^p$ topology, towards the Gaussian field $F$. Denote by $W_n$ the support of $F_n$. The Kergin interpolation $(\Pi^k_{\x})^{\otimes d}$ restricts to a $p$-interpolator from $W_n$ to $(\Pt_p^d(\R))^d$. A slight modification of the argument in the proof of Lemma \ref{lem:subsequence} shows that the constants in Lemma \ref{lem:bound} are uniformly bounded in $n$:
\be 
\inf_{n\geq0} c_n>0 \quand \sup_{n\geq0} C_n<+\infty,
\ee
Therefore $\E\qwe\# Z(\F_n,\BB)^p\ewq$ converges to $\E\qwe\# Z(\F,\BB)^p\ewq.$
\end{proof}
\subsection{Proof of Theorem \ref{thm3} and Theorem \ref{thm4}}
\begin{proof}
By covering the compact subset $K$ of $M$ with a finite family of charts, both theorems are reduced to the case in which $M$ is an open subset $\U$ of $\R^d$ equipped with a smooth metric $g$. In the following, $F$ is a Gaussian field from $\R^d$ to $\R^{d'}$, with $d'=d-n$, satisfying the hypotheses of Theorem \ref{thm3}. The case $n=0$ in both Theorem \ref{thm3}, and Theorem \ref{thm4} follows directly from Theorem \ref{thm1} and Theorem \ref{thm2}, respectively, since the cardinality of a subset of $\U$ is independent of the metric. 

Let the Grassmannian $G_n(\R^d)$ be the set of all $T\subset \R^d$ subspaces of dimension $n.$ For every point $x\in\R^d$ and $T\in G_n(\R^d)$, we define 
\be 
\mu(T,x):=\sqrt{\det g(x)|_T}.
\ee
Clearly, $\mu$ is a continuous function on $\U\times G_n(\R^d)$, therefore it is bounded by a positive constant $C$ on $K\times G_n(\R^d)$. For every $n$-dimensional submanifold $Z\subset \U,$ we have 
\be \label{eq:eucl}
\vol_g^n(Z\cap K)=\int_K\mu(T_xZ,x)
d\vol^n(x)\le C\vol^n(Z\cap K).
\ee
where $\vol^n=\vol^n_{\mathbbm{1}}$ is the standard $n$-volume measure in $\R^d.$ We are now reduced to the case where $F$ is an $\R^{d'}$-valued Gaussian field, from an open subset $\U$ of $\R^d$ endowed with the usual Euclidean metric.\jump

We now show that we can reduce to the case $n=0$. To this end, we complete the field $F$ to a random field $G=(F,\varphi_1,\dots,\varphi_n)$ from $\R^d$ to $\R^d$, so that we could apply Theorem \ref{thm1} to $G$. Observe that then, in general, 
\be\label{eq:ZFZG}
Z(F,K)\cap  \{\varphi_1=\dots =\varphi_n=0\}=Z(G,K).
\ee
In the following, $\varphi_1,\dots,\varphi_n$ are $n$ independent copies of the real Bargmann-Fock field $\varphi\colon \R^d\to \R$, defined in Section \ref{sec:BF}, also independent from $F$. We consider the auxiliary Gaussian field $G=(F,\varphi_1,\dots,\varphi_n)$ from $\U$ to $\R^d$. Observe that $\varphi$ has covariance function \be 
\E\kop\varphi(x)\varphi(y)\pok=e^{-\frac{|x-y|^2}{2}}=1-\frac{1}{2}(x-y)^T\one(x-y)+o(|x-y|^2),
\ee
so that the Adler-Taylor metric (see \cite[sec. 12.2]{AdlerTaylor} or \cite[page 173]{Aza09}) of $\varphi$ is exactly the standard Euclidean metric $\one$ on $\R^d$ (this is the reason why we choose the Bargmann-Fock field). Denote by $\E_\varphi$ the expectation with respect to the random variables $\varphi_1,\dots,\varphi_n$. By the Kac-Rice formula, for any $n$-dimensional submanifolds $Z\subset \U,$ we have the equality
\be \label{eq:Crofton}
\vol^n(Z\cap K)=v_n\E_\varphi\qwe \# \tyu Z\cap K\cap  \{\varphi_1=\dots =\varphi_n=0\}\uyt\ewq,
\ee
where $v_n=\frac12\vol^n(S^{n}),$ see \cite[Appendix B.1]{stec2022GeometrySpin}. We will use the identity \eqref{eq:Crofton} in the case when $Z=Z(F,\U)$, so that $Z\cap K=Z(F,K)$ and the identty \eqref{eq:ZFZG} holds almost surely.
Let us denote by $\E_F , \E_G=\E_F\E_\varphi$ the expectations with respect to $F,G$, respectively.
Notice that the field $G$ satisfies the hypotheses of Theorem \ref{thm1}, so that $\E_G\qwe \#  Z(G,K)^p\ewq<\infty$.
Moreover, by Jensen inequality we have that
\begin{equation}
\label{eq:jen}
\E_\varphi\qwe\#  Z(G,K)\ewq^p
\leq 
\E_\varphi\qwe \#  Z(G,K)^p\ewq.
\end{equation}
 We conclude by using, in this order, the inequalities \eqref{eq:eucl}, \eqref{eq:Crofton}, \eqref{eq:jen} and Theorem \ref{thm1} for $G$:
\bega 
\E_F\qwe \vol_g^n\tyu Z(F,K)\uyt^p\ewq &\le 
C^p\E_F\qwe  \vol^n\tyu Z(F,K)\uyt^p\ewq 
\\
&=
C^pv_n^p\E_F\qwe \E_\varphi\qwe\#  Z(G,K)\ewq^p\ewq
\\
&\le 
C^pv_n^p\E_G\qwe \#  Z(G,K)^p\ewq
\\
&<\infty.
\eega

For the holomorphic case we can rehearse the same argument, using $\varphi_\C$ and the holomorphic version of Theorem \ref{thm1}. 
\end{proof}
\begin{remark}
The identity \eqref{eq:Crofton} in the last proof is a special case of  \cite[Th. 9.9]{MathiStec} and is comparable to the classical Crofton formula, see \cite{Al07}. For more details abut such comparison see \cite[Sec. 9]{MathiStec} and  \cite[Sec. 5.1]{KRStec}.
\end{remark}
\begin{acknowledgements}
\noindent The authors are grateful to Giovanni Peccati for his careful reading of the paper and for the many comments which helped to improve the exposition. 
\end{acknowledgements}
\printbibliography

@article{AL_cursedpaper,
     author = {Ancona, Michele and Letendre, Thomas},
     title = {Multijet bundles and application to the finiteness of
moments for zeros of Gaussian fields},
     journal = {Hal preprint, hal-04165218},
     year = {2023},
}

@article{AL_rootsKost,
     author = {Ancona, Michele and Letendre, Thomas},
     title = {Roots of {Kostlan} polynomials: moments, strong {Law} of {Large} {Numbers} and {Central} {Limit} {Theorem}},
     journal = {Annales Henri Lebesgue},
     pages = {1659--1703},
     publisher = {\'ENS Rennes},
     volume = {4},
     year = {2021},
     doi = {10.5802/ahl.113},
     language = {en},
     url = {https://ahl.centre-mersenne.org/articles/10.5802/ahl.113/}
}

@article{LETEuler,
title = {Expected volume and Euler characteristic of random submanifolds},
journal = {Journal of Functional Analysis},
volume = {270},
number = {8},
pages = {3047-3110},
year = {2016},
issn = {0022-1236},
doi = {https://doi.org/10.1016/j.jfa.2016.01.007},
url = {https://www.sciencedirect.com/science/article/pii/S0022123616000173},
author = {Thomas Letendre}
}

@article {Al07,
    AUTHOR = {\'{A}lvarez Paiva, J. C. and Fernandes, E.},
     TITLE = {Gelfand transforms and {C}rofton formulas},
   JOURNAL = {Selecta Math. (N.S.)},
  FJOURNAL = {Selecta Mathematica. New Series},
    VOLUME = {13},
      YEAR = {2007},
    NUMBER = {3},
     PAGES = {369--390},
      ISSN = {1022-1824},
   MRCLASS = {53C65},
  MRNUMBER = {2383600},
       DOI = {10.1007/s00029-007-0045-5},
       URL = {https://doi.org/10.1007/s00029-007-0045-5},
}

@article{Sinha,
title={Manifold-theoretic compactifications of configuration spaces},
author={Sinha, Dev P},
journal={Selecta Mathematica},
volume={10},
pages={391--428},
year={2004},
publisher={Springer}}

@book{Dav75,
  title={Interpolation and Approximation},
  author={Davis, P.J.},
  isbn={9780486624952},
  lccn={75002568},
  series={Dover Books on Mathematics},
address={Dover},
  url={https://books.google.lu/books?id=2PaJAwAAQBAJ},
  year={1975},
  publisher={Dover Publications}
}

@article{Nie12,
title = {Scattered data interpolation and applications: A tutorial and survey},
author = {Nielson, RFGM and Franke, R},
journal = {Geometric Modeling: Methods and Their Applications},
pages = {131-160},
year={2012}}

@book {Hirsch,
    AUTHOR = {Hirsch, Morris W.},
     TITLE = {Differential topology},
    SERIES = {Graduate Texts in Mathematics},
    VOLUME = {33},
      NOTE = {Corrected reprint of the 1976 original},
 PUBLISHER = {Springer-Verlag},
address={New York},
      YEAR = {1994},
     PAGES = {x+222},
      ISBN = {0-387-90148-5},
   MRCLASS = {57-01 (58-01)},
  MRNUMBER = {1336822},
}

@book{lee2003introduction,
  title={Introduction to Smooth Manifolds},
  author={Lee, J.M.},
address= {Seattle, WA, USA},
  isbn={9780387954486},
  lccn={2002070454},
  series={Graduate Texts in Mathematics},
  url={https://books.google.nl/books?id=eqfgZtjQceYC},
  year={2003},
  publisher={Springer}
}

@book {BCR,
    AUTHOR = {Bochnak, Jacek and Coste, Michel and Roy,
              Marie-Fran{\c{c}}oise},
     TITLE = {Real algebraic geometry},
address={Berlin},
    SERIES = {Ergebnisse der Mathematik und ihrer Grenzgebiete (3)},
    VOLUME = {36},
 PUBLISHER = {Springer-Verlag},
      YEAR = {1998},
     PAGES = {x+430},
      ISBN = {3-540-64663-9},
   MRCLASS = {14Pxx (11E25 32C05 58A07)},
  MRNUMBER = {1659509},
MRREVIEWER = {A. Tognoli},
       DOI = {10.1007/978-3-662-03718-8},
       URL = {http://dx.doi.org/10.1007/978-3-662-03718-8},
}

@article {SodinTsirelson,
     AUTHOR = {Sodin, M. and Tsirelson, B.},
      TITLE = {Random complex zeroes. {I}. {A}symptotic normality},
    JOURNAL = {Israel J. Math.},
   FJOURNAL = {Israel Journal of Mathematics},
     VOLUME = {144},
       YEAR = {2004},
      PAGES = {125--149},
       ISSN = {0021-2172},
    MRCLASS = {60F05 (60D05 60G15)},
   MRNUMBER = {2121537},
MRREVIEWER = {M. Iosifescu},
        DOI = {10.1007/BF02984409},
        URL = {https://doi.org/10.1007/BF02984409},
}

@misc{stec2022GeometrySpin,
  author = {Lerario, A. and Marinucci, D. and Rossi, M. and Stecconi, M.},
  title = {Geometry and topology of spin random fields},
archivePrefix={arXiv},
  year = {2022},
  copyright = {arXiv.org perpetual, non-exclusive license},
eprint={2207.08413},
      primaryClass={math.PR},
url={https://arxiv.org/abs/2207.08413}
}

@article {GaWe2,
    AUTHOR = {Gayet, Damien and Welschinger, Jean-Yves},
     TITLE = {Betti numbers of random real hypersurfaces and determinants of
              random symmetric matrices},
   JOURNAL = {J. Eur. Math. Soc. (JEMS)},
  FJOURNAL = {Journal of the European Mathematical Society (JEMS)},
    VOLUME = {18},
      YEAR = {2016},
    NUMBER = {4},
     PAGES = {733--772},
      ISSN = {1435-9855},
   MRCLASS = {14Pxx (14Fxx 32U40 60B20)},
  MRNUMBER = {3474455},
       DOI = {10.4171/JEMS/601},
       URL = {http://dx.doi.org/10.4171/JEMS/601},
}

@article {NazarovSodin2016,
    AUTHOR = {Nazarov, F. and Sodin, M.},
     TITLE = {Asymptotic laws for the spatial distribution and the number of
              connected components of zero sets of {G}aussian random
              functions},
   JOURNAL = {Zh. Mat. Fiz. Anal. Geom.},
  FJOURNAL = {Zhurnal Matematichesko\u\i \ Fiziki, Analiza, Geometrii. Journal
              of Mathematical Physics, Analysis, Geometry},
    VOLUME = {12},
      YEAR = {2016},
    NUMBER = {3},
     PAGES = {205--278},
      ISSN = {1812-9471},
   MRCLASS = {60G15},
  MRNUMBER = {3522141},
MRREVIEWER = {Giacomo Aletti},
       DOI = {10.15407/mag12.03.205},
       URL = {http://dx.doi.org/10.15407/mag12.03.205},
}

@article{AnderssonPassare,
title = {Complex Kergin interpolation},
journal = {Journal of Approximation Theory},
volume = {64},
number = {2},
pages = {214-225},
year = {1991},
issn = {0021-9045},
doi = {https://doi.org/10.1016/0021-9045(91)90076-M},
url = {https://www.sciencedirect.com/science/article/pii/002190459190076M},
author = {Mats Andersson and Mikael Passare}
}

@book {AdlerTaylor,
    AUTHOR = {Adler, R. J. and Taylor, J. E.},
     TITLE = {Random fields and geometry},
    SERIES = {Springer Monographs in Mathematics},
 PUBLISHER = {Springer},
address={New York},
      YEAR = {2007},
     PAGES = {xviii+448},
      ISBN = {978-0-387-48112-8},
   MRCLASS = {60G60 (58J65)},
  MRNUMBER = {2319516},
MRREVIEWER = {Jos{\'e} Rafael Le{\'o}n},
}

@misc{dtgrf,
    title={Differential Topology of {G}aussian Random Fields},
    author={Lerario, A. and Stecconi, M.},
doi={},
url={https://arxiv.org/abs/1902.03805},

publisher = {arXiv},

    year={2019},
    eprint={1902.03805},
    archivePrefix={arXiv},
    primaryClass={math.DG}
}

@article{KRStec,
	author = {Stecconi, M.},
	date = {2022/02/21},
	doi = {10.1007/s13324-022-00654-0},
	id = {Stecconi2022},
	isbn = {1664-235X},
	journal = {Analysis and Mathematical Physics},
	number = {2},
	pages = {44},
	title = {Kac-{R}ice formula for transverse intersections},
	url = {https://doi.org/10.1007/s13324-022-00654-0},
	volume = {12},
	year = {2022}}

@article{stec2019MaxTyp,
      title={Maximal and Typical Topology of Real Polynomial Singularities}, 
      author={Lerario, A. and Stecconi, M.},
      year={2019},
      Journal={Ann.Inst.Fourier, in press},
eprint={1906.04444},
      archivePrefix={arXiv},
      primaryClass={math.AG}
}

@misc{MathiStec,
  doi = {10.48550/ARXIV.2210.11214},
  
  url = {https://arxiv.org/abs/2210.11214},
  
  author = {Mathis, Léo and Stecconi, Michele},
  
  title = {Expectation of a random submanifold: the zonoid section},
  
  publisher = {arXiv},
  
  year = {2022},
    eprint={2210.11214},
    archivePrefix={arXiv},
    primaryClass={math.PR}
}

@article{KERGIN,
title = {A natural interpolation of Ck functions},
journal = {Journal of Approximation Theory},
volume = {29},
number = {4},
pages = {278-293},
year = {1980},
issn = {0021-9045},
doi = {https://doi.org/10.1016/0021-9045(80)90116-1},
url = {https://www.sciencedirect.com/science/article/pii/0021904580901161},
author = {Paul Kergin}
}

@article{MICCHELLI,
title = {A formula for Kergin interpolation in Rk},
journal = {Journal of Approximation Theory},
volume = {29},
number = {4},
pages = {294-296},
year = {1980},
issn = {0021-9045},
doi = {https://doi.org/10.1016/0021-9045(80)90117-3},
url = {https://www.sciencedirect.com/science/article/pii/0021904580901173},
author = {Charles A Micchelli and Pierre Milman}
}

@article {Anc21c,
    AUTHOR = {Ancona, Michele and Letendre, Thomas},
     TITLE = {Zeros of smooth stationary {G}aussian processes},
   JOURNAL = {Electron. J. Probab.},
  FJOURNAL = {Electronic Journal of Probability},
    VOLUME = {26},
      YEAR = {2021},
     PAGES = {Paper No. 68, 81},
   MRCLASS = {60F05 (60F15 60F17 60F25 60G10 60G15 60G55 60G57)},
  MRNUMBER = {4262341},
       DOI = {10.1214/21-ejp637},
       URL = {https://doi.org/10.1214/21-ejp637},
}

@article{Arm23,
  title={On the finiteness of the moments of the measure of level sets of random fields},
  author={Armentano, Diego and Aza\"{i}s, Jean Marc and Dalmao, Federico and Le{\'o}n, Jos{\'e} Rafael and Mordecki, Ernesto},
  journal={Brazilian Journal of Probability and Statistics},
  volume={37},
  number={1},
  pages={219--245},
  year={2023},
  publisher={Brazilian Statistical Association}
}

@book {Aza09,
    AUTHOR = {Aza\"{i}s, Jean-Marc and Wschebor, Mario},
     TITLE = {Level sets and extrema of random processes and fields},
 PUBLISHER = {John Wiley \& Sons, Inc.},
address={Hoboken, NJ},
      YEAR = {2009},
     PAGES = {xii+393},
      ISBN = {978-0-470-40933-6},
   MRCLASS = {60-02 (60E15 60G05 60G15 60G60 60G70)},
  MRNUMBER = {2478201},
MRREVIEWER = {Anna Amirdjanova},
       DOI = {10.1002/9780470434642},
       URL = {https://doi.org/10.1002/9780470434642},
}

@article{Aza22,
  title={Mean number and correlation function of critical points of isotropic Gaussian fields and some results on GOE random matrices},
  author={Azaïs, Jean-Marc and Delmas, C{\'e}line},
  journal={Stochastic Processes and their Applications},
  volume={150},
  pages={411--445},
  year={2022},
  publisher={Elsevier}
}

@article {Bel18,
    AUTHOR = {Beliaev, Dmitry and Wigman, Igor},
     TITLE = {Volume distribution of nodal domains of random band-limited
              functions},
   JOURNAL = {Probab. Theory Related Fields},
  FJOURNAL = {Probability Theory and Related Fields},
    VOLUME = {172},
      YEAR = {2018},
    NUMBER = {1-2},
     PAGES = {453--492},
      ISSN = {0178-8051},
   MRCLASS = {58J50 (35P20 35R01 60F99 60G60)},
  MRNUMBER = {3851836},
MRREVIEWER = {Bruce A. Watson},
       DOI = {10.1007/s00440-017-0813-x},
       URL = {https://doi.org/10.1007/s00440-017-0813-x},
}

@article{Bel19,
  title={Two point function for critical points of a random plane wave},
  author={Beliaev, Dmitry and Cammarota, Valentina and Wigman, Igor},
  journal={International Mathematics Research Notices},
  volume={2019},
  number={9},
  pages={2661--2689},
  year={2019},
  publisher={Oxford University Press}
}

@article{Bel22,
  title={A central limit theorem for the number of excursion set components of Gaussian fields},
  author={Beliaev, Dmitry and McAuley, Michael and Muirhead, Stephen},
  journal={arXiv preprint arXiv:2205.09085},
  year={2022}
}

@article {Ber77,
    AUTHOR = {Berry, M. V.},
     TITLE = {Regular and irregular semiclassical wavefunctions},
   JOURNAL = {J. Phys. A},
  FJOURNAL = {Journal of Physics. A. Mathematical and General},
    VOLUME = {10},
      YEAR = {1977},
    NUMBER = {12},
     PAGES = {2083--2091},
      ISSN = {0305-4470},
   MRCLASS = {81.58 (28A65 58F15)},
  MRNUMBER = {489542},
       URL = {http://stacks.iop.org/0305-4470/10/2083},
}

@article{Bla19,
  title={Limit theory for geometric statistics of point processes having fast decay of correlations},
  author={B{\l}aszczyszyn, Bart{\l}omiej and Yogeshwaran, Dhandapani and Yukich, Joseph E},
  journal={The Annals of Probability},
  volume={47},
  number={2},
  pages={835--895},
  year={2019},
  publisher={Institute of Mathematical Statistics}
}

@Article{Can16,
 Author = {Yaiza {Canzani} and Boris {Hanin}},
 Title = {{Local universality for zeros and critical points of monochromatic random waves}},
 FJournal = {{Communications in Mathematical Physics}},
 Journal = {{Commun. Math. Phys.}},
 ISSN = {0010-3616; 1432-0916/e},
 Volume = {378},
 Number = {3},
 Pages = {1677--1712},
 Year = {2020},
 Publisher = {Springer, Berlin/Heidelberg},
 Language = {English},
 MSC2010 = {58J65 35P20}
}

@article {Cuz75,
    AUTHOR = {Cuzick, Jack},
     TITLE = {Conditions for finite moments of the number of zero crossings
              for {G}aussian processes},
   JOURNAL = {Ann. Probability},
  FJOURNAL = {The Annals of Probability},
    VOLUME = {3},
      YEAR = {1975},
    NUMBER = {5},
     PAGES = {849--858},
      ISSN = {0091-1798},
   MRCLASS = {60G15},
  MRNUMBER = {388515},
MRREVIEWER = {Georg Lindgren},
       DOI = {10.1214/aop/1176996271},
       URL = {https://doi.org/10.1214/aop/1176996271},
}

@article{Die20,
  title={Small scale CLTs for the nodal length of monochromatic waves},
  author={Dierickx, Gauthier and Nourdin, Ivan and Peccati, Giovanni and Rossi, Maurizia},
  journal={arXiv preprint arXiv:2005.06577},
  year={2020}
}

@article{Eke14,
  title={Recovering the good component of the Hilbert scheme},
  author={Ekedahl, Torsten and Skjelnes, Roy},
  journal={Annals of mathematics},
  pages={805--841},
  year={2014},
  publisher={JSTOR}
}

@article {Gay17,
    AUTHOR = {Gayet, Damien and Welschinger, Jean-Yves},
     TITLE = {Betti numbers of random nodal sets of elliptic
              pseudo-differential operators},
   JOURNAL = {Asian J. Math.},
  FJOURNAL = {Asian Journal of Mathematics},
    VOLUME = {21},
      YEAR = {2017},
    NUMBER = {5},
     PAGES = {811--839},
      ISSN = {1093-6106},
   MRCLASS = {58J40 (34L20 60D05)},
  MRNUMBER = {3767266},
MRREVIEWER = {Nikhil Savale},
       DOI = {10.4310/AJM.2017.v21.n5.a2},
       URL = {https://doi.org/10.4310/AJM.2017.v21.n5.a2},
}

@article {Gas21,
    AUTHOR = {Gass, Louis},
     TITLE = {Almost-sure asymptotics for Riemannian random waves},
   JOURNAL = {To appear in Bernoulli Journal},
      YEAR = {2021},
       URL = {https://arxiv.org/abs/2005.06389},
}

@misc{Gas21t,
      title={\textcolor{white}{z}\!Cumulants asymptotics for the zeros counting measure of real Gaussian processes}, 
      author={Gass, Louis},
      year={2021},
      eprint={2112.08247},
      archivePrefix={arXiv},
      primaryClass={math.PR}
}

@article {Kac43,
    AUTHOR = {Kac, M.},
     TITLE = {On the average number of real roots of a random algebraic
              equation},
   JOURNAL = {Bull. Amer. Math. Soc.},
  FJOURNAL = {Bulletin of the American Mathematical Society},
    VOLUME = {49},
      YEAR = {1943},
     PAGES = {314--320},
      ISSN = {0002-9904},
   MRCLASS = {41.1X},
  MRNUMBER = {7812},
MRREVIEWER = {P. Erd\H{o}s},
       DOI = {10.1090/S0002-9904-1943-07912-8},
       URL = {https://doi.org/10.1090/S0002-9904-1943-07912-8},
}

@article {Kri13,
    AUTHOR = {Krishnapur, Manjunath and Kurlberg, P\"{a}r and Wigman, Igor},
     TITLE = {Nodal length fluctuations for arithmetic random waves},
   JOURNAL = {Ann. of Math. (2)},
  FJOURNAL = {Annals of Mathematics. Second Series},
    VOLUME = {177},
      YEAR = {2013},
    NUMBER = {2},
     PAGES = {699--737},
      ISSN = {0003-486X},
   MRCLASS = {58C40 (60G60)},
  MRNUMBER = {3010810},
MRREVIEWER = {Nelia Charalambous},
       DOI = {10.4007/annals.2013.177.2.8},
       URL = {https://doi.org/10.4007/annals.2013.177.2.8},
}

@article{Lad22,
  title={Local repulsion of planar Gaussian critical points},
  author={Ladgham, Safa and Lachi{\`e}ze-Rey, Rapha{\"e}l},
  journal={arXiv preprint arXiv:2209.04150},
  year={2022}
}

@article{Mal94,
  title={Some finiteness conditions for factorial moments of the number of zeros of Gaussian field zeros},
  author={Malevich, Tat'yana L'vovna and Volodina, LN},
  journal={Theory of Probability \& Its Applications},
  volume={38},
  number={1},
  pages={27--45},
  year={1994},
  publisher={SIAM}
}

@article {Mar16,
    AUTHOR = {Marinucci, Domenico and Peccati, Giovanni and Rossi, Maurizia
              and Wigman, Igor},
     TITLE = {Non-universality of nodal length distribution for arithmetic
              random waves},
   JOURNAL = {Geom. Funct. Anal.},
  FJOURNAL = {Geometric and Functional Analysis},
    VOLUME = {26},
      YEAR = {2016},
    NUMBER = {3},
     PAGES = {926--960},
      ISSN = {1016-443X},
   MRCLASS = {60G60 (35J10 35P20 58J50 60B10 60D05)},
  MRNUMBER = {3540457},
       DOI = {10.1007/s00039-016-0376-5},
       URL = {https://doi.org/10.1007/s00039-016-0376-5},
}

@article {Naz09,
    AUTHOR = {Nazarov, Fedor and Sodin, Mikhail},
     TITLE = {On the number of nodal domains of random spherical harmonics},
   JOURNAL = {Amer. J. Math.},
  FJOURNAL = {American Journal of Mathematics},
    VOLUME = {131},
      YEAR = {2009},
    NUMBER = {5},
     PAGES = {1337--1357},
      ISSN = {0002-9327},
   MRCLASS = {60F10 (33C55 43A90 60D05)},
  MRNUMBER = {2555843},
       DOI = {10.1353/ajm.0.0070},
       URL = {https://doi.org/10.1353/ajm.0.0070},
}

@article{Naz12,
  title={Correlation functions for random complex zeroes: strong clustering and local universality},
  author={Nazarov, Fedor and Sodin, Mikhail},
  journal={Communications in Mathematical Physics},
  volume={310},
  number={1},
  pages={75--98},
  year={2012},
  publisher={Springer}
}

@article {Nou19,
    AUTHOR = {Nourdin, Ivan and Peccati, Giovanni and Rossi, Maurizia},
     TITLE = {Nodal statistics of planar random waves},
   JOURNAL = {Comm. Math. Phys.},
  FJOURNAL = {Communications in Mathematical Physics},
    VOLUME = {369},
      YEAR = {2019},
    NUMBER = {1},
     PAGES = {99--151},
      ISSN = {0010-3616},
   MRCLASS = {58J50 (60F05 60G60)},
  MRNUMBER = {3959555},
MRREVIEWER = {Jos\'{e} Rafael Le\'{o}n},
       DOI = {10.1007/s00220-019-03432-5},
       URL = {https://doi.org/10.1007/s00220-019-03432-5},
}

@article {Ric45,
    AUTHOR = {Rice, S. O.},
     TITLE = {Mathematical analysis of random noise},
   JOURNAL = {Bell System Tech. J.},
  FJOURNAL = {The Bell System Technical Journal},
    VOLUME = {24},
      YEAR = {1945},
     PAGES = {46--156},
      ISSN = {0005-8580},
   MRCLASS = {60.0X},
  MRNUMBER = {11918},
MRREVIEWER = {M. Kac},
       DOI = {10.1002/j.1538-7305.1945.tb00453.x},
       URL = {https://doi.org/10.1002/j.1538-7305.1945.tb00453.x},
}

@incollection {Zel09,
    AUTHOR = {Zelditch, Steve},
     TITLE = {Real and complex zeros of {R}iemannian random waves},
 BOOKTITLE = {Spectral analysis in geometry and number theory},
    SERIES = {Contemp. Math.},
    VOLUME = {484},
     PAGES = {321--342},
 PUBLISHER = {Amer. Math. Soc.},
ADDRESS	= {Providence, RI},
      YEAR = {2009},
   MRCLASS = {58J51 (35J05)},
  MRNUMBER = {1500155},
MRREVIEWER = {Julie Rowlett},
       DOI = {10.1090/conm/484/09482},
       URL = {https://doi.org/10.1090/conm/484/09482},
}
\end{document}